\documentclass{article}
\usepackage{mthesis_paper}
\usepackage[ngerman]{babel}
\usepackage{tcolorbox}

\title{Numerical Approximation of  Optimal Convex Shapes in $\mathbb{R}^3$}

\author{S\"oren Bartels\footnote{Universität Freiburg, Hermann-Herder-Str. 10, 79104, Freiburg, bartels@mathematik.uni-freiburg.de},\ Hedwig Keller\footnote{Corresponding Author, Universität Freiburg, Hermann-Herder-Str. 10, 79104, Freiburg,hedwig.keller@mathematik.uni-freiburg.de} \ and Gerd Wachsmuth\footnote{BTU Cottbus-Senftenberg, Postfach 10 13 44, 03013, Cottbus, gerd.wachsmuth@b-tu.de}}

\begin{document}

\maketitle
\selectlanguage{english}
\pagenumbering{arabic}

\abstract{In the optimization of convex domains under a PDE constraint numerical difficulties arise in the approximation of convex domains in $\mathbb{R}^3$. Previous research used a  restriction to rotationally symmetric domains  to reduce shape optimization problems to a two-dimensional setting, \cite{keller}. 
In the current research, two approaches for the approximation in $\mathbb{R}^3$ are considered. First, a notion of discrete convexity allows for a  nearly convex approximation with polyhedral domains.
An alternative  approach is based on the  recent observation that higher order finite elements can approximate convex functions conformally, \cite{wachsmuth}.
 As a second approach these results are used to approximate optimal convex domains with isoparametric convex domains.
The proposed algorithms were tested on shape optimization problems constrained by a Poisson equation and both algorithms achieved similar results.}

\paragraph{Acknowledgment:} \ \\This article will appear as part of a collection of articles of the work achieved by the Priority Program SPP 1962 "Non-smooth and Complementarity-based Distributed Parameter Systems: Simulation and Hierarchical Optimization" and summarizes the research within the Project P1 "Approximation of Non-Smooth Optimal Convex Shapes with Applications in Optimal Insulation and Minimal Resistance". Support by DFG grants BA2268/4-2 within the Priority Program SPP 1962 is gratefully acknowledged.

\section{Introduction}

Solvability of shape optimization problems relies, among other factors, on
suitable constraints on the geometry of the admissible domains.
Since we minimize over shapes, no topology is readily available. The
restriction to classes of convex domains appears attractive, since the
compactness results available for convex domains let us avoid more general
topological frameworks. 
For corresponding analytical details we refer to
\cite{BD03,variationalmethods,  BG, veodf, vangoethem, yang09}.
Therefore, we restrict the shape optimization to  open, convex and bounded domains. 
 
The numerical approximation of convex domains is difficult in higher
 dimensions. Indeed, for conformal P1 finite elements we can not guarantee that  a convex function can be approximated consistently (cf. \cite{chone}), and with simple  examples we can show that the nodal interpolant of a convex function is not  necessarily convex itself, for such an example see \cite[Figure 2.1]{aguilera}.
 To approximate convex functions  we need for example higher order conforming finite elements, \cite{wachsmuth}, a weaker definition for convexity tailored to finite  elements, \cite{aguilera}, a geometric approach as in  \cite{convexhull} or a spherical harmonic decomposition, \cite{antunes}.
 Since the approximation of convex domains in $\mathbb{R}^3$
 has certain similarities to the approximation  of convex functions in $\mathbb{R}^2$,  we expect related difficulties. We focus on the framework of \cite{BW18}, which approximates optimal convex domains in $\mathbb{R}^2$ and gives an example for the numerical difficulties in higher dimensions. 
 
In previous research, \cite{keller}, a restriction to   a class of rotationally symmetric domains in $\mathbb{R}^3$ allowed us to reduce the problem to a two-dimensional setting, for which the boundary is a convex curve. The restriction to rotational symmetric domains is in general not justified and the impact of this restriction is difficult to evaluate. 
Therefore, we propose two approaches to deal with the convexity constraint for domains in $\mathbb{R}^3$ with less geometric restrictions.

We cannot expect a conforming approximation of convex domains with an affine tetrahedralization in $\mathbb{R}^3$. We therefore define a relaxed convexity condition, the \textit{discrete convexity}, tailored to P1 finite elements, based on \cite{CL_nonconf}.  A similar method was recently studied in \cite{rieger}, by considering a Galerkin approach to the approximation of convex domains. The results however are more abstract and less easy to adapt to PDE-constraint shape optimization problems such as the one considered in our research.
From \cite{wachsmuth} we know that we can approximate convex functions conformally with higher order finite elements. This motivates our second approach, for which we approximate optimal convex domains with convex isoparametric domains with a piecewise quadratic boundary.

We are interested in the optimization under a PDE constraint.  For a more general background on shape optimization we refer to \cite{variationalmethods,veodf}. As in \cite{BW18} we consider optimization problems in which the state equation is the Poisson equation and search for a solution of the following problem:
\begin{align*}
\begin{cases}
\text{ Minimize} & \ \ \ \int_{\Omega}j(x,u(x),\nabla u(x))\, \text{d} x \\ \text{ w.r.t.} & \ \ \ \Omega \subset \mathbb{R}^3, u \in H_0^1(\Omega) \\ \text{ s.t.} & \ \ \ - \Delta u = f \text{ in } \Omega \text{ and } \Omega \subset Q \text{ convex and open.}
\end{cases}
\end{align*}
Here, $Q \subset \mathbb{R}^3$ is bounded, open and convex, $j : Q \times \mathbb{R} \times \mathbb{R}^3 \rightarrow \mathbb{R}$ a suitable Carath\'{e}odory  function and $f \in L^2(Q)$. The numerical approximation of  this problem in $\mathbb{R}^2$, as well as a proof of existence is covered in \cite{BW18}.
In Section \ref{p1:sec:non_conf} a discrete convexity constraint is proposed and a shape optimization algorithm is implemented for affine tetrahedralizations. The following Section \ref{p1:sec:chd} covers a conformal approximation of convex domain with isoparametric domains. Lastly, in Section \ref{p1:sec:conclusion} we compare the approaches and  discuss their limitations.

\section{Discrete polyhedral convexity}\label{p1:sec:non_conf}

It is well known that we cannot expect that a convex function can be approximated by convex piecewise node patchs. 
In \cite{CL_nonconf} an external approach is suggested instead. The idea is to approximate convex functions with continuous piecewise affine linear functions, but with the  nodal interpolants of convex functions instead of convex P1 functions.
This motivates us to define a relaxed discrete convexity for polyhedral domains.

\subsection{Discrete convexity}

A natural way to adapt the results of \cite{CL_nonconf} to convex domains is the following definition.

\begin{definition}\label{p1:def:discrete_convex_domain}\rm
Let $\Omega_h \subset \mathbb{R}^3$ be a polyhedral bounded domain with a triangulation $\mathcal{T}_h$ with maximal mesh size $h$. Then $\Omega_h$ is \textit{discrete convex} if $\partial \Omega_h$ is a piecewise linear Lagrange interpolant of the boundary $\partial \Omega$ of a convex domain $\Omega$.
\end{definition}
For details on the construction of an affine linear approximation of the boundary, see for example \cite{ciarlet,dziukelliot}. Using the minimality of the convex hull an equivalent defintion can be derived, which is more convenient for implementation purposes. 
\begin{definition}\rm\label{p1:def:discrete_convex_domain_support}
Let $\Omega_h \subset \mathbb{R}^3$ be a polyhedral bounded domain with a triangulation $\mathcal{T}_h$ with maximal mesh size $h$. We say that $\Omega_h$ is \textit{discrete convex}, if every boundary vertex $z \in \partial \mathcal{N}_h = \mathcal{N}_h \cap \partial \Omega_h$ has a support plane for $\Omega_h$.
\end{definition}

\begin{figure}[b]
\begin{center}
	\includegraphics[trim = {1cm 1cm 1cm 1cm}, height =3.5cm]{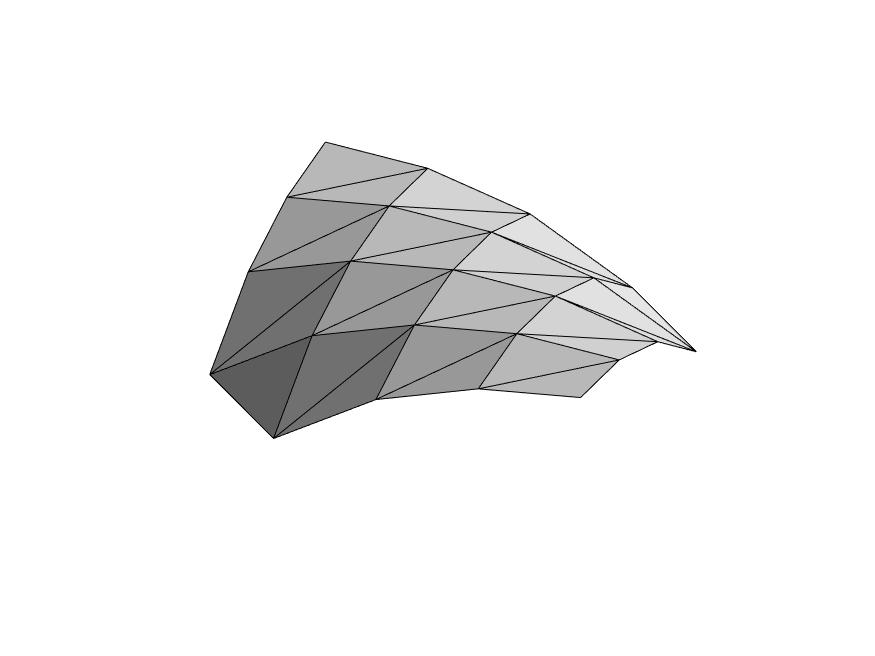}
     \includegraphics[trim={3cm 4cm 6cm 2cm},height = 3.5cm]{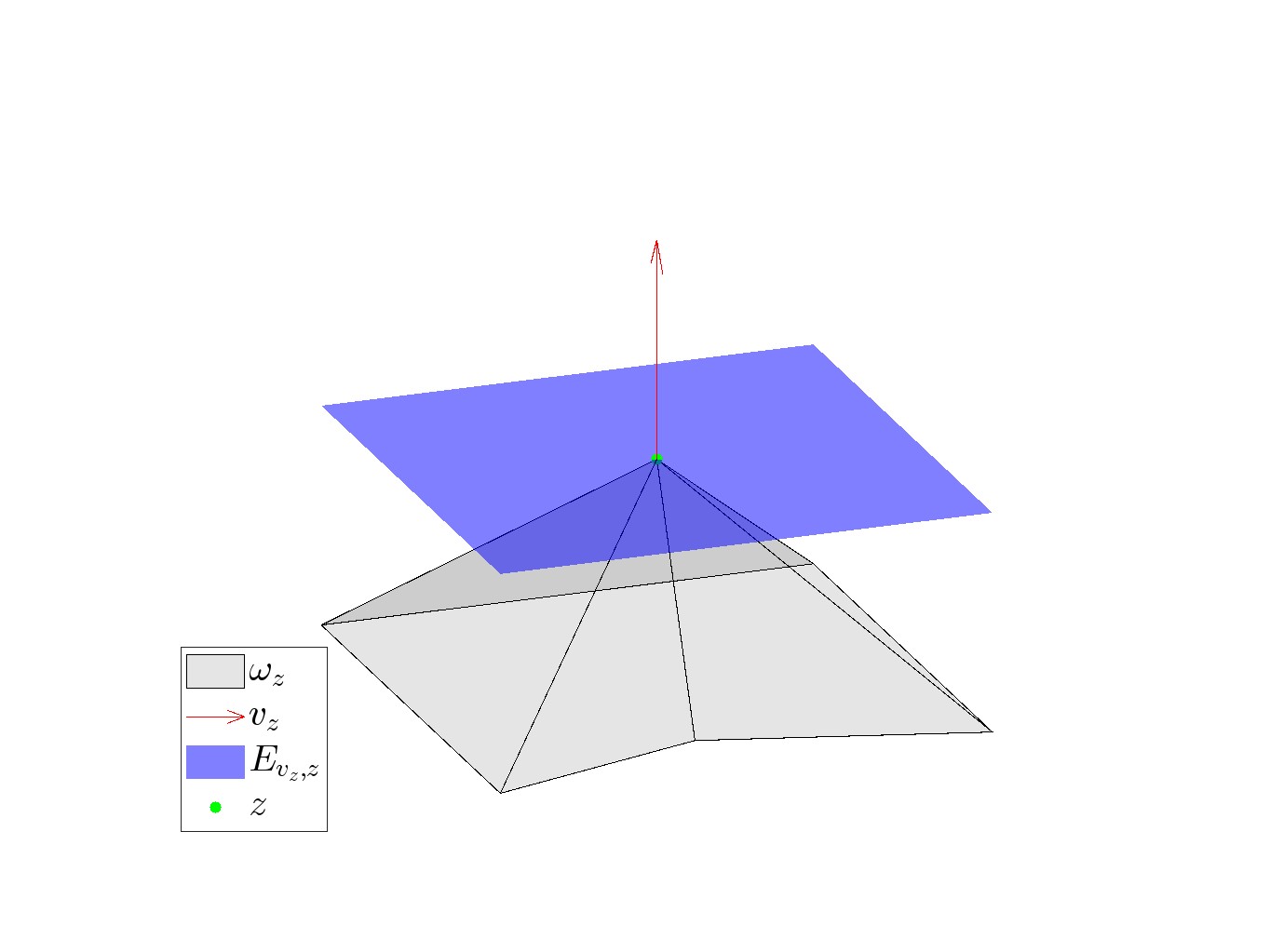}
  \caption{Example for Definition \ref{p1:def:discrete_convex_domain} (left), a boundary segment of discrete convex approximation of a cylinder, and of Definition \ref{p1:def:discrete_convex_domain_support} (right): a boundary segment $\omega_z$ has a support plane $E_{v_z,z} = \{x \in \mathbb{R}^3: (z-x)\cdot v_z = 0\}$ at vertex $z$, defined by a normal vector $v_z$} \label{p1:fig:dc_support_sketch}
  \end{center}
		
     \end{figure}  
Examples for both Definitions \ref{p1:def:discrete_convex_domain} and \ref{p1:def:discrete_convex_domain_support} are shown in Figure \ref{p1:fig:dc_support_sketch}.
In general, the polyhedral domain interpolating the boundary of a convex domain is not convex. Instead, by definition it is discrete convex. On the other hand, a discrete convex domain always interpolates a polyhedral convex domain, for example its convex hull.

\begin{corollary}\label{p1:corollary:ldc:dist_convhull}
For every discrete convex domain $\Omega_h$  there exists a convex domain $\Omega$ such that $\partial \mathcal{N}_h \subset \partial \Omega$ and $\vert \chi_{\Omega_h} - \chi_{\Omega} \vert \le c\vert \partial \Omega \vert h$.
\end{corollary}

\begin{proof}
We can show that a discrete convex domain interpolates its convex hull $\Omega$.
Since $\Omega_h \subset \Omega$  we have $\vert \Omega \backslash \Omega_h \vert \le c\vert \partial \Omega \vert \sup_{x \in \partial \Omega_h} \text{dist}(x,\partial \Omega)$.  But, dist$(x,\partial \Omega) = \inf_{y \in \partial \Omega} \vert x-y\vert \le \inf_{z \in \partial \mathcal{N}_h} \vert x-z \vert \le h$, for the maximal mesh size $h$ of the triangulation $\mathcal{T}_h.$ 
This implies the asserted estimate.
\end{proof}

Standard results on boundary approximation, cf. \cite[Section 3.6.]{dziuk} imply the following proposition.

\begin{proposition}
For $\Omega \subset \mathbb{R}^3$ convex, open and bounded with a piecewise $C^2$-boundary there exists a sequence of discrete convex domains $\Omega_h\subset \Omega$ with uniformly shape regular triangulations with maximal mesh size $h$ such that $\Omega_h \Delta \Omega \subset \{x \in \mathbb{R}^3: \textup{dist}(x,\partial \Omega) \le ch^2 \}.$ Especially we have $\chi_{\Omega_h} \rightarrow \chi_\Omega$ in $L^1$.
\end{proposition}
In general, we have the following result.
\begin{corollary}
For $\Omega  \subset Q \subset \mathbb{R}^3$ convex, open and bounded there exists a sequence of discrete convex domains $\Omega_h\subset \Omega$ with uniformly shape regular triangulations with maximal mesh size $h$ such that $\chi_{\Omega_h} \rightarrow \chi_\Omega$ in $L^1(Q)$.
\end{corollary}

For the shape optimization, the compactness results for convex domains is especially relevant. A similar result can be derived for the approximation with discrete convex domains.
\begin{corollary}[Compactness]Let $(\Omega_h)_{h >0}$ be a sequence of discrete convex domains with quasiuniform triangulations $(\mathcal{T}_h)_{h>0}$ such that $\Omega_h\subset Q\subset \mathbb{R}^3$, with $Q$ convex, bounded and open. There exists a convex domain $\Omega$, such that for a subsequence $\chi_{\Omega_h} \rightarrow \chi_{\Omega}$ in $L^1(Q)$ for $h \rightarrow 0$. \label{p1:thm:compactness_nonconf}
\end{corollary}
\begin{proof}
The domains $\widetilde{\Omega}_h = \text{conv}(\Omega_h)$  define a sequence of bounded, convex domains. After passing to a subsequence, there exists a convex domain $\Omega$, s.t. $\chi_{\widetilde{\Omega}_h} \rightarrow \chi_{\Omega}$ in $L^1(Q)$, see \cite{BG97}. 
We have that
\begin{equation*}
\Vert \chi_{\Omega_h} - \chi_{\Omega} \Vert_{L^1(Q)} \le  \Vert \chi_{\Omega_h} - \chi_{\widetilde{\Omega}_h} \Vert_{L^1(Q)}  +\Vert \chi_{\widetilde{\Omega}_h} - \chi_{\Omega} \Vert_{L^1(Q)}  \rightarrow 0.
\end{equation*}
The second term goes to zero due the convergence results for convex domains. For the first term we have with the estimate of Corollary \ref{p1:corollary:ldc:dist_convhull} and the quasiuniform triangulatinos that 
\begin{equation*}
\Vert \chi_{\Omega_h} - \chi_{\widetilde{\Omega}_h} \Vert_{L^1(Q)} \le C\vert \partial \widetilde{\Omega}_h \vert h \le C\vert \partial Q\vert h \rightarrow 0,
\end{equation*}
since for two convex domains $A \subset B$ holds that $\vert \partial A \vert \le \vert \partial B \vert$. 
\end{proof}

The strength of a conformal approximation of convex domains are the convergence results regarding the boundary, like the compactness of the characteristic functions with respect to the convergence in variation. For the discrete convex domain, we only approximate the volume, such that we only gain convergence with respect to the $L^1$-norm. 
For now we only consider shape optimization problems constrained by a Poisson equation, for which this convergence is sufficient.

\subsection{Shape optimization for the Poisson problem}

We look at the following PDE constrained shape optimization problem with convexity constraint:
\begin{align*} \label{p1:problem:P}
\begin{cases}
\text{ Minimize} & \ \ \ \int_{\Omega}j(x,u(x),\nabla u(x)) \text{d} x \tag{$\mathbf{P}$} \\ \text{ w.r.t.} & \ \ \ \Omega \subset \mathbb{R}^3, u \in H_0^1(\Omega) \\ \text{ s.t.} & \ \ \ - \Delta u = f \text{ in } \Omega \text{ and } \Omega \subset Q \text{ convex and open.}
\end{cases}
\end{align*}
Here, $Q \subset \mathbb{R}^3$ is bounded, open and convex and $j : Q \times \mathbb{R} \times \mathbb{R}^3 \rightarrow \mathbb{R}$ a suitable Carath\'{e}odory  function and $f \in L^2(Q)$.
\begin{proposition}[\cite{BW18}] 
Assume that $j$ is a Carath\'{e}odory  function and that there exist $a \in L^1(Q)$ and $c \ge 0$ such that $
\vert j(x,s,z) \vert \le a(x) + c(\vert s\vert^p + \vert z \vert^2)$ with $p = 2d/(d-2)$ for $d \ge 2$. Then there exists an optimal pair $(\Omega,u)$ for $\mathbf{(P)}$.\label{p1:prop:pp_existence}
\end{proposition}
We consider the class $ \mathbb{T}$ of conforming, uniformly shape regular triangulations $\mathcal{T}_h$ of polyhedral subsets of $\mathbb{R}^3$ with elements $T \in \mathcal{T}_h$ with diameter $h_T \le h$, and impose a non-conformal discrete convexity constraint for the discrete problem.
\begin{align*} \label{p1:problem:P_h}
\begin{cases}
\text{ Minimize} & \ \ \ \int_{\Omega_h}j(x,u_h(x),\nabla u_h(x)) \text{d} x  \tag{$\mathbf{P_h}$}\\ \text{ w.r.t.} & \ \ \ \Omega_h \subset \mathbb{R}^3, \mathcal{T}_h \in \mathbb{T} \text{ triangulation of } \Omega_h, u_h \in \mathcal{S}_0^1(\mathcal{T}_h) \\ \text{ s.t.} & \ \ \ - \Delta_h u_j = f_h \text{ in } \Omega_h \text{ and } \Omega_h \subset Q \text{ discrete convex and open.}
\end{cases}
\end{align*}

Since we use an affine approximation of the domains and P1 finite elements, there are already many results available concerning the discretization and convergence analysis of the approximation of PDEs and the shape optimization. With the compactness of discrete convex domains, see Corollary \ref{p1:thm:compactness_nonconf}, the results from \cite{BW18} can be adapted to prove the stability of the shape optimization algorithm.

\begin{proposition}[Stability, cf. Proposition 4.3, \cite{BW18}]
Assume that j satisfies the conditions of Proposition \ref{p1:prop:pp_existence}, and such that  the
functional J defined by
\begin{equation*}
J(A, v) = \int_A j(x, v(x), \nabla v(x))\textup{d}x
\end{equation*}
satisfies for some constant $c_{J} > 0$ the estimate
\begin{equation*}
\vert J(A,v) -J(A,w)\vert \le c_J (1 + \Vert \nabla v \Vert_{L^{2}(A)} + \Vert \nabla w \Vert_{L^{2}(A)})\Vert \nabla (v-w) \Vert_{L^{2}(A)}
\end{equation*}
for every open set $A \subset Q$ and $v,w \in H^{1}_0(\Omega)$. 
Let $(\Omega_h, \mathcal{T}_h,u_n)_{h>0}$ be a sequence of discrete solutions of \eqref{p1:problem:P_h}. Then every accumulation pair $(\Omega,u)$ solves \eqref{p1:problem:P}.
\end{proposition}

\subsection{Algorithmic realization}

We discuss in this subsection aspects of the numerical realization of the shape optimization under the constraint of discrete convexity, such as the derivation of a local discrete convexity constraint in Section \ref{p1:sec:ldc} and a description of the shape optimization algorithm in Section \ref{p1:sec:ldc:so_alg}.
\subsubsection{Localized convexity constraint}\label{p1:sec:ldc}

Let $\mathcal{T}_h$ be a mesh of the discrete domain $\Omega_h$ with nodes $\mathcal{N}_h$. We will use Definition \ref{p1:def:discrete_convex_domain_support} of discrete convexity: The domain $\Omega_h$  is discrete convex if for any boundary node $z \in \mathcal{N}_h \cap \partial \Omega_h=:X$ there exists a support plane of $\Omega_h$. In order to decrease the number of constraints, we enforce the discrete convexity only locally and require each boundary node to have a support plane for the tetrahedral node patch $\omega^3_z= \bigcup_{T \in \mathcal{T}_h, z \in T} T$. In practice, we only look at the triangular node patch $\omega_z = \bigcup_{F \in \partial \mathcal{T}_h,z \in F} F$ of the boundary facet and only use the boundary nodes to search for a support plane, but such that the normal points into the exterior of the domain $\Omega_h$.  The convexity might then be violated globally, so we include a post processing to guarantee the discrete convexity globally. 

We formulate the local discrete convexity (l.d.c.) in the following way: for each node patch $\omega_z$, there exists an outward pointing normal $n_z$, which defines a halfspace $H_{n_z,z } = \{x: (z-x) \cdot n_z \ge 0 \}$, such that all other points of the node patch lie $H_{n_z,z }$. 
For the entire node patch we denote the set of admissible normals for a boundary node $z$, i.e. pointing into the exteriour of $\Omega_h$, as $N_z$. 
Then  $\Omega_h$ is locally discrete convex, if for every node patch $\omega_z$ of $\partial \Omega_h$, around the node $z$ with nodes boundary $z_i$, $i = 1,\dots m_z$ with normal $n_z \in N_z$
\begin{equation}
\widetilde{C}^z_i(X,n_z) = (z_i-z) \cdot n_z \le 0 \ \ \forall i = 1,\dots m_z.
\end{equation}
The functions $\widetilde{C}^{z}_i$ are differentiable, however the variables $n_z$ are in general not known, and the l.d.c. is satisfied if there exists a vector $n_z\in N_z $ such that $ \widetilde{C}^z_i(X,n_z) \le 0$, which can  no longer  be described by a differentiable function.
We discuss next how to construct the set $N_z$ and determine the variables $n_z(X)$. We then instead consider the derivative $D_X \widetilde{C}(X,n_z(X))$. 
It would be preferable to compute the convexity constraint independently of $n_z$, rather then computing suitable variables $n_z(X)$, but at this point it is not clear how to avoid this.
\begin{enumerate}
\item For a fixed point $z \in X$, we define the set $N_z$ of admissible normals. For this, we construct a a convex cone $C_z$ with vertex $z$ which satisfies the property that every support plane for the node patch $\omega^3_z$ at the vertex $z$ is a support plane for $C_z$. Then we define the set of admissible normals $N_z$ as all vectors $v_z$ which define a support plane for $C_z$, i.e.
\begin{equation}\label{p1:eq:czsubhz}
H_{v_z,z} = \{x: (z-x)\cdot v_z \ge 0\} \supseteq C_z.
\end{equation} 
The convexity of $C_z$ implies $N_z \neq \emptyset$, and any vector defining a support plane for $\bigcup_j T _j$ has to lie in $N_z$.

 If a vector $v_z$ exists such that the l.d.c is satisfied, it needs to define a support plane for each $T_i$, and each $T_i$ has at $z$ has a support plane defined by $n_{F_i}$.
With this, we construct the cone $C_z$ from the halfspaces defined by $n_{F_i}$, i.e.
\begin{equation}\label{p1:eq:definition_cz}
C_z = \bigcap_{i = 1,\dots m_z} H_{n_{F_i},z} = \bigcap_{i = 1,\dots m_z} \{x \subset \mathbb{R}^3: (z-x)\cdot n_{F_i} \ge 0\}.
\end{equation}
As a finite intersection of convex sets, $C_z$ is convex itself,and therefore also locally discrete convex in $z$, in particular every $n_{F_i}, i = 1, \dots m_z$ defines a support plane.
The set $C_z$ is optimal in the sense, that if $\bigcup_j T_j$ is locally convex in $z$, then every vector in $N_z$ defines a support plane for $\bigcup_j T_j$.

We note, that the local convexity of $\bigcup_{j} T_j$ is not implied by the local discrete convexity at $z$.
\item A more convenient description of the set $N_z$ can be found using Farkas' Lemma, which states that the intersection of finitely many halfspaces defines a polyhedral cone.

Without loss of generality assume $z = 0$, and $A = [n_{F_i}]_{i = 1,\cdots m_z} \in \mathbb{R}^{3,m_z}$. Then with Farkas' Lemma, \cite[Lemma 2.27]{geiger}, we have that the following equivalence of the following statements:
\begin{enumerate}
\item There exists $ \lambda \ge 0 \text{ such that } v_z= A \lambda = \sum_{i} \lambda_i n_{F_i}. $
\item $\text{ For any } x  \text{ such that } A^\top x \ge 0 \text{ i.e. }  x \in \bigcap H_{n_{F_i},z} \Rightarrow  x^\top v_z \ge 0 \text{, i.e. } x \in H_{v_z,z}.$
\end{enumerate}

We can therefore define equivalently $N_z = \{\sum_{i} \lambda_i n_{F_i}: \lambda_i \in \mathbb{R}^{m_z}_{\ge 0}\}.$
 
\item We assume the outer vertices of the node patch are ordered counter clockwise, such that each facet is given by $F_i = \text{conv}(z_i,z_{i+1},z)$ for $i = 1,\dots m_z$, with $z_{m_z+1} = z_1$.
Then, in order to check if the condition of l.d.c. is satisfied for a node patch $\omega_z$, with outer vertices $z_i$, we need to find a non-vanishing vector $v_z$ which defines a support plane for all $T_i$ and satisfies
\begin{equation}
(z_i-z) \cdot v_z \le 0 \quad \text{ for all } i = 1,\dots m_z\label{p1:eq:v_1}.
\end{equation}

With the second definition of $N_z$ we arrive at the equivalent condition for the l.d.c.: For the nodes, $z,z_i$, there exists $\lambda = [\lambda_i]_{i} \in \mathbb{R}^{m_z}_{\ge 0 }\backslash \{0\}$, such that
\begin{align}
(z_i-z) \cdot \sum_{j}\lambda_j n_{F_j} = M\lambda \le 0 \quad &\text{ for all } i = 1,\dots m_z\label{p1:eq:v_11}
\end{align}
with $M \in \mathbb{R}^{m_z\times m_z}$ and $ M_{i,j} = (z_i-z) \cdot n_{F_j}$. \\
Including the constraint $\sum_{i}\lambda_i  = 1$ allows us to make sure that the solution is not trivial, and $I\lambda \ge \varepsilon$, for  $\varepsilon >0$ to ensure that the vector $v_z$ does not lie on the boundary of the polyhedral cone (depending on the dimension of the cone), i.e. that no outer unit normal of the facets in the node patch is neglected.
We find a feasible point by solving the following linear program for a  suitable $k \in \mathbb{N}$:
\begin{align}\label{p1:eq:ldc_linear_programm}
\begin{cases}
\min_{\lambda} 1\cdot M\lambda &\quad \text{s.t. } M\lambda \le 0 \\ \text{and }\sum_{i}\lambda_i = 1,&\quad I \lambda \ge 0.5^k\varepsilon
\end{cases}
\end{align}
We can expect that the linear programm \eqref{p1:eq:ldc_linear_programm} has a feasible point as $k \rightarrow \infty$, if the l.d.c. is satisfied, since any vector defining a support plane for the node patch lies in $N_z$. 
Assume we have found $v_z = \sum_j \lambda_j n_{T_j} = \sum_j \lambda_j (z -z_j) \times (z-z_{j+1})$. 
We then have
\begin{equation}
C^z_i(X) = \sum_{j}\lambda_j(z_i-z) \cdot [(z-z_j) \times (z-z_{j+1})]\le 0, \quad i = 1,\dots m_z.
\end{equation}
\item Now, if we can find a feasible point of \eqref{p1:eq:ldc_linear_programm}, we have found a support plane at $z$, such that each $z_i \in H_{v_z,z}$ for all $i = 1,\dots m_z$. Since we only consider affine linear non-degenerate triangulations $\mathcal{T}_h$, this implies that $v_z$ defines a support plane for the whole node patch, i.e. the l.d.c. is satsified. 
\end{enumerate}

For each boundary node $z \in \partial \mathcal{N}_h$ we can then find a $ \lambda \ge 0 \in \mathbb{R}^{m_z}$ for which we can then derive a linearised convexity constraint
\begin{equation}
C^z_i(X) + DC^z_i(X)\cdot V \le 0, \quad i =1,\dots, m_z
\end{equation}
 which can  be included in the shape optimization algorithm.

\subsubsection{Shape optimization algorithm}\label{p1:sec:ldc:so_alg}

In this section we explain our ideas for the shape optimization algorithm. 
Since the shape optimization algorithm is similar to the one in \cite{keller,BW18}, we focus on particular aspects here.

We iteratively update the domain in  a descent direction for our objective value which we want to minimize.
After evaluating the convexity of the current domain and the state equation, we search for a suitable deformation field that acts as a descent direction. For this, we compute a shape derivative. In our case, with an affine linear triangulation and the Poisson equation as the state equation, the shape derivative is well researched and a volume formulation is available, which works well for the discrete problem. For analytic details, we refer to \cite{veodf,so_shape_sensitivity_analysis}, while for the numerical details we refer to \cite{hiptmair}.
The volume formulation of the shape derivative is given by
\begin{align*}
J^\prime(\Omega;V) = &\int_{\Omega} j_x(\cdot) \cdot V - j_v(\cdot) \cdot DV^\top \nabla u + \nabla p^\top[-DV -DV^\top + \text{div}(V)I]\nabla u \\ &- \text{div}(fV)p + j(\cdot)\text{div}(V) \, \text{d}x
\end{align*}
with the adjoint state $p \in H^1_0(\Omega)$, which solves
\begin{equation}
\int_{\Omega} \nabla p \cdot \nabla w \, \text{d}x = - \int_{\Omega} j_u(\cdot) w + j_v(\cdot) \cdot \nabla w\, \text{d}x
\end{equation}
for all $w \in H^1_0(\Omega)$ with  the partial derivatives $j_x,j_u,$ and $j_v$. From \cite{hiptmair}, we know that for the problem under consideration the derivative and discretization commute in the sense of forming the Eulerian derivative and that a similar expression is available for the P1 finite element discretization.

From the shape derivative we compute the deformation field with an appropriate inner product. Here, we employ the bilinar form motivated by a problem of linear elasticity   as in \cite{BW18} or \cite[Section 3]{schulz}
\begin{equation}
a_h(V_h,W_h) : = \int_{\Omega_h} 2\mu \varepsilon(V_h) : \varepsilon(W_h) + \lambda \text{div}(V_h) \text{div}(W_h) + \delta V_h \cdot W_h\,\text{d}x
\end{equation}
with homogenous Neumann boundary conditions for $V_h, W_h  \in \mathcal{S}^1(\mathcal{T}_h)^3$, with the symmetric gradient $\varepsilon(V_h) := (\nabla V_h + \nabla V_h^\top)/2$, and  $\mu, \lambda>0$ the Lam\'{e} parameters, and $\delta >0$ a damping parameter to ensure coercivity of the bilinear form.
We then search for a deformation field $V_h$, which solves
\begin{equation}
a_h(V_h,W_h) =-J_h^\prime(\Omega_h; W_h) \quad \text{ for all } W_h \in \mathcal{S}^1(\mathcal{T}_h)^3 
\end{equation}
for the shape derivative $J_h^\prime(\Omega_h; \cdot)$.
With Korn's inequality, solvability and discretization with P1 finite elements is straightforward, cf. \cite[Chapter 8]{pde}.
The choice for the damping and Lam\'{e}  parameters is  discussed in the respective numerical experiments in Section \ref{p1:ldc:sec:num_ex}.

Including the l.d.c. constraint we solve the problem 
\begin{align*}
\begin{cases}
\text{Min } & \frac{1}{2} a_h(V_h,V_h) +J^\prime_h(\Omega_h,V_h)\\
\text{w.r.t. }& V_h \in \mathcal{S}^1(\mathcal{T}_h)^3 \\
\text{s.t. } & C^z_i(X) +t^0 DC^z_i(X)V_h(X) \le 0 \quad \text{for all } z \in \partial \mathcal{N}_h \text{ and } i = 1, \dots m_z,
\end{cases}
\end{align*}
using the solver OSQP, \cite{stellato2020osqp}.\\
In order to ensure that the objective value does not increase, and to control the deformation of the domain and any violation of the local discrete convexity constraint, a line search is included, and the domain then updated accordingly. Lastly, to guarantee that the discrete convexity is not violated globally, a post processing is included.
These steps are then repeated, until $\tau = \vert J_h^\prime(\Omega_h,V_h) \vert < \varepsilon_{\text{stop}}$, such that the current iteration is almost stationary.

\subsection{Numerical experiments} \label{p1:ldc:sec:num_ex}

We are interested in the approximation of the following problem:
\begin{align*}
\begin{cases}
\text{ Minimize} & \ \ \ J(\Omega) = \int_{\Omega}u \ \text{d} x \\ \text{ w.r.t.} & \ \ \ \Omega \subset Q \subset \mathbb{R}^3 \text{ convex and open}, u \in H^1(\Omega)  \\ \text{ s.t.} & \ \ \  -\Delta u +u= f \text{ in } \Omega,  \quad \partial_n u = 0 \text{ on } \partial \Omega.
\end{cases}
\end{align*}
In our experiments we use the local discrete convexity constraint in the discretization and consider two particular choices for the function $f$:
\begin{align}
f_1(x_1,x_2,x_3)& = x_1^2+x_2^2+x_3^2-1\label{p1:P1} \tag{$\mathbf{P_1}$}\\
f_2(x_1,x_2,x_3) & = 20(x_1+0.4-x_2^2)^2+x_1^2+x_2^2+x_3^2-1 \label{p1:P2}\tag{$\mathbf{P_2}$}
\end{align}

\begin{figure}[b] 
\begin{center}
     \includegraphics[height = 3cm]{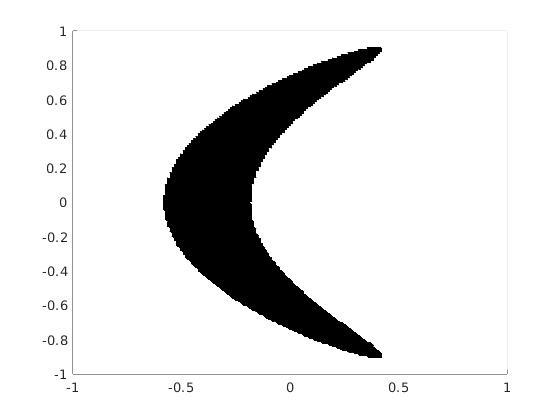}
     \end{center}
  \caption{Intersection at $x_3 = 0$ of non-convex sublevel set $\{x \in \mathbb{R}^3: f_2(x) < 0 \}$,}\label{p1:fig:level_set_nonconvex}
     \end{figure}

From the functions we expect that the optimal domain for  \eqref{p1:P1} is the ball, since the sublevel set $\{f_1<0\} = B_1(0)$ is already convex. The second problem \eqref{p1:P2} has less symmetry and the convexity constraint is not redundant, since  the sublevel set $\{f_2<0\}$ defines a non-convex set, see Figure \ref{p1:fig:level_set_nonconvex}.

We denote by $E_{a,b}(x)$ the ellipsoid with center $x$ with the radii $(a,b,c)$, with $c = 1/ \sqrt{ab}$, such that $E_{a,b}$ has the volume of the unit ball. For example $E_{1,1}(0) = B_1(0)$.
We have
\begin{equation}
E_{a,b}(x) = \left\{ x \in \mathbb{R}^3: (x_1/a)^2+(x_2/b)^2+(x_3/c)^2 -1 < 0;\, c = 1/ab \right\}.
\end{equation}
As initial domains, we use the discretizations of the ellipsoids $\Omega^0_1 = E_{0.8,0.8},\Omega_2^0 = E_{1,1} = B_1(0)$,  and $\Omega_3^0 = E_{1.1,1.1}$, with refinements $h = 2^{-\ell},\ell = 3$. The initial domains are shown in Figure \ref{p1:fig:pp_so_elastic_p2}.
By construction, we expect that the domain $\Omega_2^0$ is already nearly optimal for the problem \ref{p1:P1}.
As elasticity parameters we choose $E = 0.5,\nu =0.2$ for Young's Module, and the damping parameter $\rho = 0.5$.

The results of the shape optimization of problem \ref{p1:P1} are listed in Table \ref{p1:table:pp_so_elastic}, for initial domains $\Omega^0_i , i= 1,2,3$ and refinement $h = 2^{-\ell},\ell = 3$. The optimized domains are not shown here, but as expected, it is an approximation of the unit ball. In this experiment, the discrete convexity constraint is redundant and in experiments without convexity constraints similar results can be observed, see \cite[Section 6.3]{BW18}. 

Next we look at the optimization of problem \ref{p1:P2}. Here, the convexity constraint has an impact on the optimization compared to proplem \ref{p1:P1} due to the non-convexity of the sublevel set $\{f_2 < 0\}$. 
The results of the shape optimization are listed in Table \ref{p1:table:pp_so_elastic}, for initial domains $\Omega^0_i , i= 1,2,3$ and refinements $\ell = 3$. The optimized domain for $\ell = 3$ are displayed in Figure \ref{p1:fig:pp_so_elastic_p2}. As we can see, the optimized domains shrink such that they are almost discrete convex and nearly contained inside the non-convex domain, for which $f_2$ is the sublevel set function.  The results for both experiments are consistent with our expectations, and with the experiments of our second approach, see Section \ref{p1:sec:isop_ex}.

\begin{table}[!t]
\caption{Objective values for \ref{p1:P1} (top) and \ref{p1:P2} (bottom) for the initial domains $\Omega_i^0$, approximated optimal domains $\Omega_i$, as well as the volume of initial and optimized domain, the violation of convexity, stopping criteria $\tau$, and number of iterations $k$, $i = 1,2,3$, and refinement level  $\ell = 3$}\label{p1:table:pp_so_elastic}
\begin{center}
\begin{tabular}{ p{1cm} p{1.5cm}p{1.5cm}p{1.5cm}p{1.5cm}p{1.5cm}p{1.5cm}p{0.5cm}}
\hline\noalign{\smallskip}
 \ref{p1:P1} & $J_1(\Omega_i^0)$ & $J_1(\Omega_i )$ & $\vert \Omega_i^0\vert$ & $\vert \Omega_i \vert$ & $\tau$ & $\max C(\Omega_i)$ & $k$ \\ \noalign{\smallskip}\hline\noalign{\smallskip}
$ i = 1 $& -1.5289	& -1.6437	& 4.1494 &	4.0924	& -0.0118 &	-0.0445	& 3\\
$i = 2$ & -1.6462	& -1.6462	& 4.1526 &	4.1526	& -1.181e-18 &	-0.0539	& 1\\
$i = 3$ & -1.6227	& -1.6459	& 4.1518	& 4.1351	& -0.0025 &	-0.0447	& 3\\
\noalign{\smallskip}\hline\noalign{\smallskip}
& & & & & & &  \\
\hline\noalign{\smallskip}
 \ref{p1:P2} & $J_2(\Omega_i^0)$ & $J_2(\Omega_i )$ & $\vert \Omega_i^0\vert$ & $\vert \Omega_i \vert$ & $\tau$ & $\max C(\Omega_i)$ & $k$ \\ \noalign{\smallskip}\hline\noalign{\smallskip}
$ i = 1 $&	17.0783	&-0.2093&	4.1494&	0.6266&	-0.1275&	0.0014	&90\\
$ i = 2$&	22.2511&	-0.2086&4.1526&	0.6312&	-0.1454&	0.0014	&90\\
$ i = 3$ &25.7313	&-0.2077&	4.1518&	0.6303&	-0.1924&	0.0014	&94\\
\noalign{\smallskip}\hline\noalign{\smallskip}
\end{tabular}
\end{center}
\end{table}

\begin{figure}[t] 
\begin{center}
 \includegraphics[trim = {10cm 0 9cm 0},height =4cm]{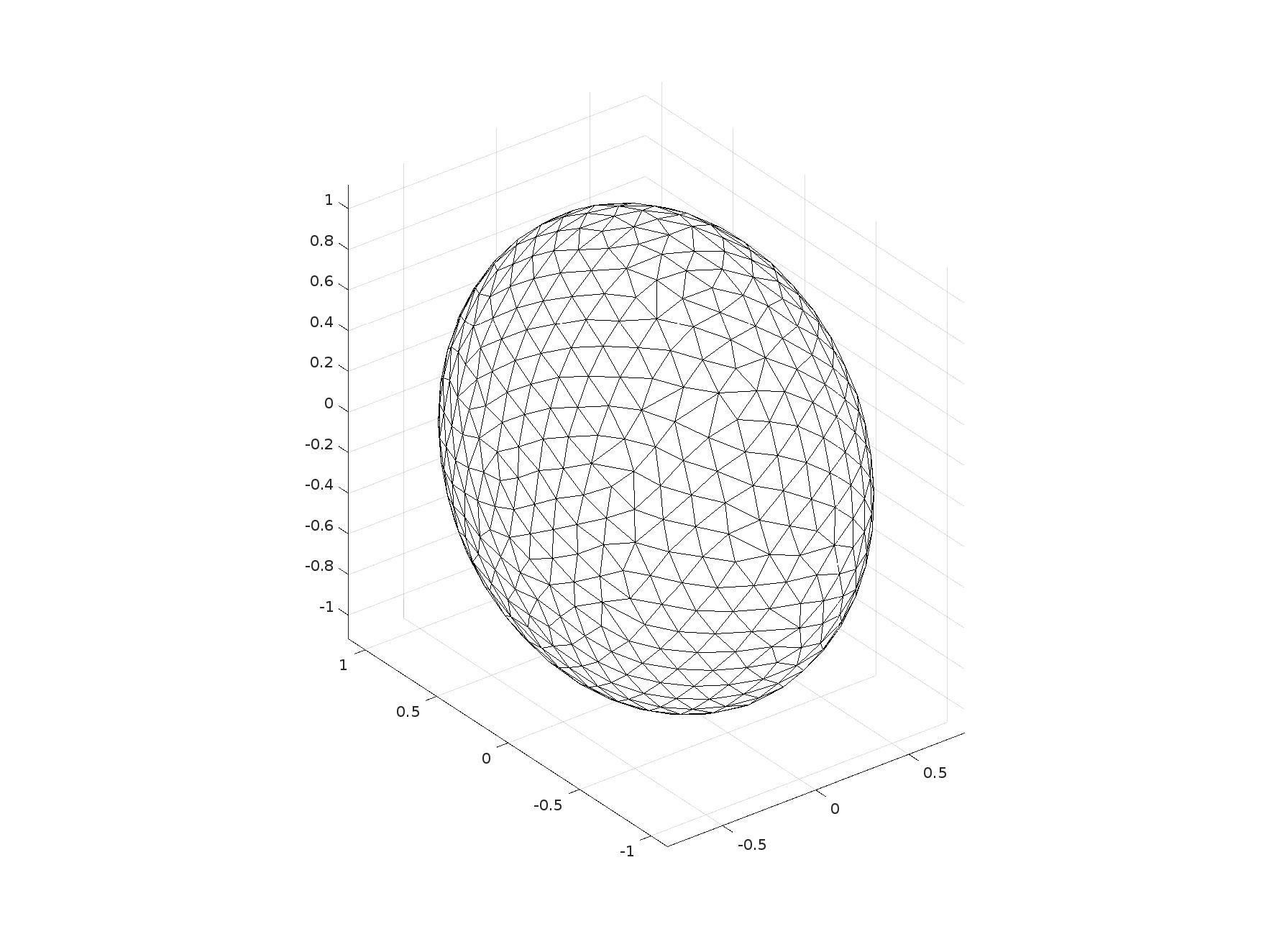}
  \includegraphics[trim = {10cm 0 8cm 0},height =4cm]{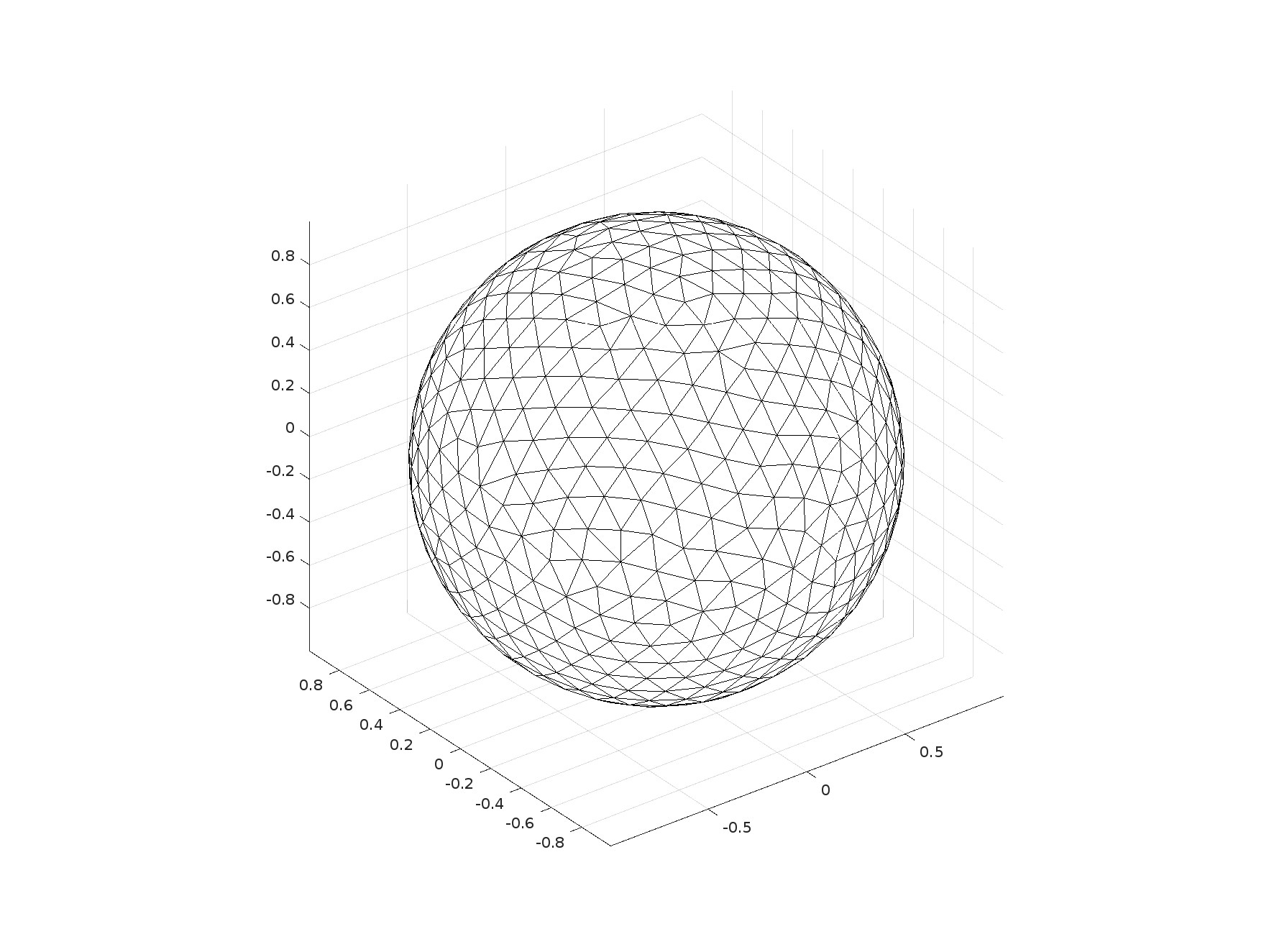}
   \includegraphics[trim = {10cm 0 10cm 0},height =4cm]{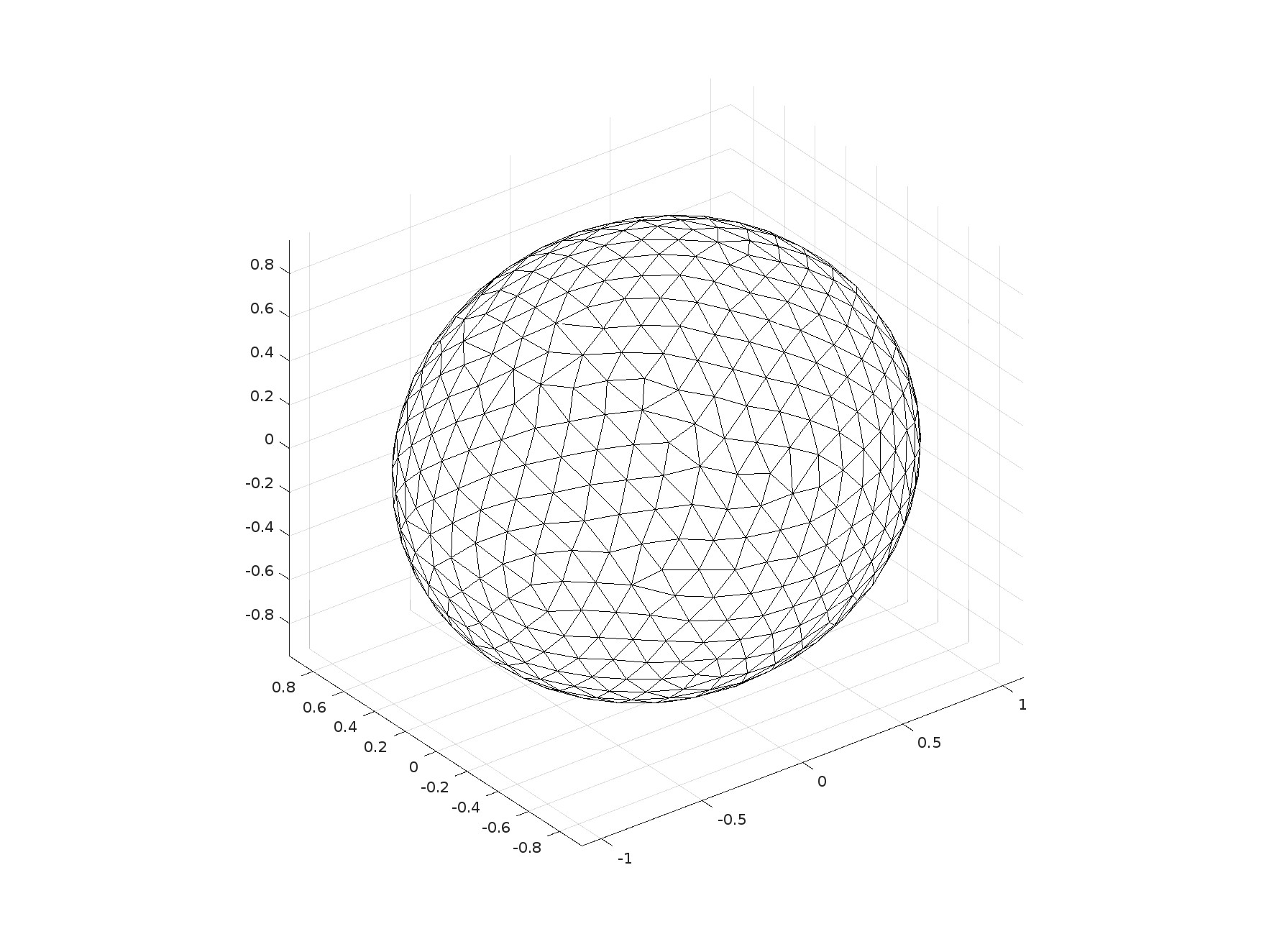}
\\
     \includegraphics[trim = {7cm 0 1cm 0},height =4cm]{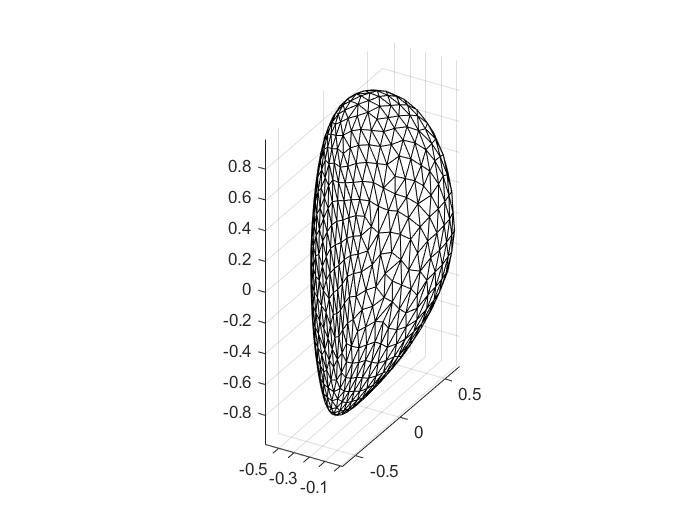}
     \includegraphics[trim = {7cm 0 1cm 0},height = 4cm]{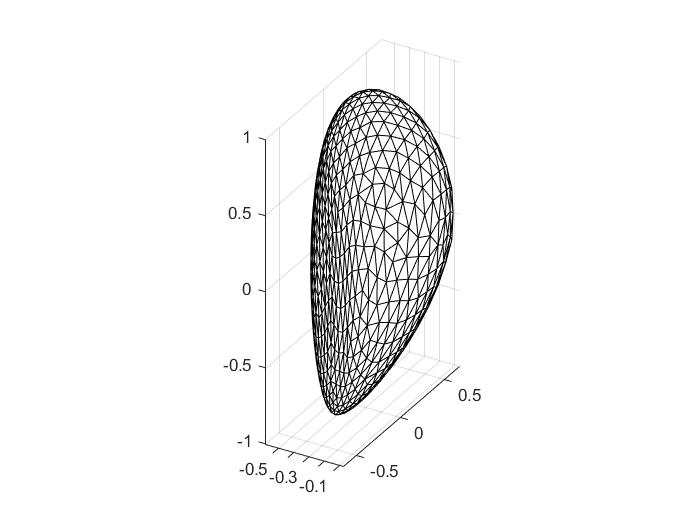}
     \includegraphics[trim = {7cm 0 5cm 0},height = 4cm]{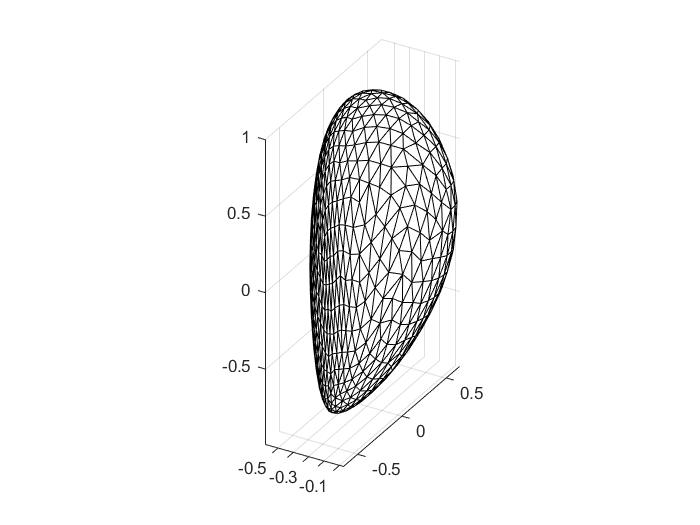}
  \caption{Initial (top) and optimized  (bottom) domains for problem \ref{p1:P2} for initial domain $\Omega^0_i, i = 1,2,3$, with refinement parameter $\ell = 3$}\label{p1:fig:pp_so_elastic_p2}
\end{center}
     \end{figure}

\section{Isoparametric approximation of convex domains}\label{p1:sec:chd}

In the previous section we adressed a non-conformal approximation with polyhedral domains. The definition of discrete convexity relies only on the approximation of the volume. The regularity of the boundary of convex sets is lost and the compactness result, especially regarding the convergence in variation, is not obvious. This poses a problem for optimization problems with boundary terms such as the eigenvalue arising in a problem in optimal insulation, \cite{BB19,keller,BBN17}. Therefore a conformal approximation of convex sets is still desirable.

Previous results from \cite{wachsmuth} show that convex functions can be conformally approximated by higher order finite elements. This motivates approximating convex domains with isoparametric finite elements  with a piecewise quadratic boundary. 
We summarize the main results from \cite{wachsmuth} in Section \ref{p1:sec:wachsmuth} followed by brief introduction to isoparametric finite elements in Section \ref{p1:sec:background_isop}. Then we apply these results to convex domains in Section \ref{p1:sec:isop_approx}. In Section \ref{p1:sec:isop_so} we summarize the shape optmization problem constrained by a Poisson equation, and discuss the shape optimization algorithm in Section \ref{p1:sec:isop_alg}.  Lastly in Section \ref{p1:sec:isop_ex}, we show the results of the numerical experiments of the shape optimization algorithm.

\subsection{Conformal approximation of convex functions} \label{p1:sec:wachsmuth}

We shortly describe the results from \cite{wachsmuth}.
Assume we have a domain $\omega \subset \mathbb{R}^2$ with a triangulation $\widehat{\mathcal{T}}$. Let $\mathcal{P}^k(\widehat{\mathcal{T}})$ be the continuous piecewise polynomials of degree at most $k$.
We assume that $\{\widehat{\mathcal{T}}_h\}_{h >0}$ is a family of  quasi-uniform triangulations.
\begin{theorem}[Theorem 2.6, \cite{wachsmuth}]
A function $u_h \in \mathcal{P}^k(\widehat{\mathcal{T}}_h)$ is convex if and only if 
\begin{align}
\nabla^2u_h(x) \succeq 0 & \quad \text{ for all } x \in \textup{int}(T), T \in \widehat{\mathcal{T}}\label{p1:eq:hessian_condition}\\
[\![ \nabla u_h^\top n_F ]\!] \ge 0& \quad \text{ for all } x \in \textup{relint}(F), F \in \mathcal{F}_i(\widehat{\mathcal{T}}) \label{p1:eq:kink_condition}
\end{align}
hold. Here $n_F$ is the normal vector of the facet $F$ and $\mathcal{F}_i(\widehat{\mathcal{T}})$ is the set of interior facets of the triangulation $\widehat{\mathcal{T}}_h$.
\end{theorem}
Under some assumptions on the triangulation we can approximate a  convex function $u \in W^{3,\infty}(\omega)$ with a convex piecewise quadratic function $u_h$. 
The interpolant
\begin{equation}
u_h = \mathcal{I}^2_h(u) +\gamma_1 h \psi_h + \gamma_2 h\phi \in \mathcal{P}^2(\widehat{\mathcal{T}}_h)
\end{equation}
 is composed of  the Lagrange interpolant $\mathcal{I}^2_h(u)$ in $\mathcal{S}^2(\widehat{\mathcal{T}}_h)$, a linear function $\psi_h \in \mathcal{P}^1(\widehat{\mathcal{T}}_h)$ to correct the gradient jumps to ensure \eqref{p1:eq:kink_condition} is satisfied, and the quadratic function $\phi(x) = (1/2)\Vert x \Vert^2$, to ensure that  \eqref{p1:eq:hessian_condition} is satisfied, with  appropriately chosen $\gamma_1, \gamma_2 \ge 0$, see \cite[Theorem 3.2]{wachsmuth} for details. 
 
Such an interpolant exists under the assumption on the triangulation that a suitable correction function $\psi_h \in \mathcal{P}^1(\widehat{\mathcal{T}}_h)$ exists, see \cite[Assumption 3.1]{wachsmuth}.
This assumption is satisfied for example if the triangulations $\widehat{\mathcal{T}}_h$ are uniformly acute, or the triangulations $\widehat{\mathcal{T}}_h $ are $\delta$-protected with $\delta \ge c h$ with $c > 0$ independent of $h$, i.e. if there lies no other node within the disctance $\delta$ of the circumcircle of each triangle in $\mathcal{T}_h$.

\begin{corollary}[\cite{wachsmuth}, Theorem 3.2 and Corollary 3.3]
Assume that the assumption \cite[Assumption 3.1]{wachsmuth} on the triangulation is satisfied and that $u \in H^1(\omega)$ is convex. Then there exist convex functions $u_h \in P^2(\widehat{\mathcal{T}}_h)$, such that, as $h \rightarrow 0$, 
\begin{equation}
\Vert u -u_h \Vert_{H^1(\omega)} \rightarrow 0.
\end{equation}
\end{corollary}
With higher regularity, such as $u \in W^{3, \infty}(\omega)$, we can obtain the convergence rate $\mathit{O}(h)$. This convergence rate is suboptimal in the absence of a convexity condition, but cannot be improved in general, see \cite[Section 5.2]{wachsmuth}.

In what follows we investigate if convex domains can be approximated in a similar fashion, with a piecewise quadratic approximation of the boundary. For convex domains which can be described as the hypograph of concave functions, such that the convexity condition can be applied globally, the application of the results to convex domains is obvious. If the convexity condition has to be applied to local parametrizations of the boundary of a convex domain, an equivalent construction of an interpolant and the interpretation of the conditions on the triangulation is more difficult.

\subsection{Isoparametric finite elements} \label{p1:sec:background_isop}

Since we want to approximate the boundary with a piecewise quadratic function, we approximate $\Omega$ with isoparamteric finite elements. In our approach, we only consider isoparametric finite elements with a piecewise quadratic boundary and omit this specification for simplicity.
We follow the notation of \cite{pde} and provide a short overview of the main concepts. For more details on isoparametric finite elements we refer to \cite[Section 4.3]{ciarlet} and \cite{demlow} for the approximation of the boundary with higher order elements. 

For an isoparametric triangulation $\mathcal{T}^3_h$ of $\Omega$ a given $T^3 \in \mathcal{T}^3_h$ is specified by a  quadratic function $a_{T^3}$, s.t. $T^3 = a_{T^3}(K_3)$ for the reference tetrahedron $K_3= \text{conv}(\{0,e_1,e_2,e_3\}) \subset \mathbb{R}^3$, with the canonical basis vector $e_i, i =1,\dots,3$, i.e.
\begin{equation}
a_{T^3}:K_3 \rightarrow T^3, \quad a_{T^3}(\widehat{x}) \mapsto \sum_{j =1}^{4} D_j^{T^3} \widehat{\varphi}_j(\widehat{x}) + \sum_{1 \le i < j\le 4} D_{i,j}^{T^3} \widehat{\varphi}_{i,j}(\widehat{x}).
\end{equation}
with the hierarchical basis of functions $\widehat{\varphi}_k$ of $\mathcal{P}^2(K_3)$ for $ k \in N:= \{1,\dots, 4\} \cup \{(i,j),1\le i < j \le 4\}$.

For $i = 1,\dots 4, \, D_i^{T^3}$ corresponds to the vertices $P_i \in \mathbb{R}^3$ of ${T^3}$, and $D_{i,j}^{T^3} = P_{i,j}^{T^3} - (P_i^{T^3}+P_j^{T^3})/2$ for $1\le i < j \le 4$. We then have for the nodes $z_k$ of ${T^3}$ that $a_{T^3}(\widehat{z}_k) = P_k^{T^3}$ for $ k \in N$.

For $\Omega$ we have the nodes $\mathcal{N}_h = \{ z \in \mathbb{R}^3: {T^3} \in \mathcal{T}^3_h, k \in N, z = P_k^{T^3}\} $. On $\Omega$ we use the basis functions $\varphi_z$ for each node $z \in \mathcal{N}_h$
\begin{equation}
\varphi_z(x) = \begin{cases} \widehat{\varphi}_k \circ  a_{T^3}^{-1}(x) &\text{ for } a_{T^3}(\widehat{z}_k) = z \\ 0 & \text{otherwise.} \end{cases}
\end{equation}
With this we define the finite element space
\begin{equation}
\mathcal{S}^{2,iso}(\mathcal{T}^3_h) = \left\{ v_h \in C(\overline{\Omega}): v_h = \sum_{z \in \mathcal{N}_h} \alpha_z \varphi_z \right \}.
\end{equation}
By $\widehat{T}^3$ we denote the corresponding affine simplex, i.e. the image of the map $\widehat{a}_{T^3} = \sum_{j=1}^4  D_j^{T^3}\widehat{\varphi}_j$, and analogously $\widehat{\mathcal{T}}_h$ the affine triangulation. 
We denote the nodal interpolant for the isoparametric interpolant to the transformed basis as $\mathcal{I}^{2,iso}$.

Interpolation results as well as approximation results  for the Poisson problem or homogeneous second order Dirichlet problems, are available see e.g. \cite[Sections 4.3 and 4.4]{ciarlet}.

\subsection{Application to convex domains} \label{p1:sec:isop_approx}

Next, we discuss how to approximate convex domains with a convex isoparametric domain. Usually, the interpolation of a surface $\Gamma \subset \mathbb{R}^3$ with curved elements is straightforward, \cite{demlow}. However, ensuring that the surface corresponds to a convex domain is more difficult. Adapting the results from \cite{wachsmuth} to convex domains is only possible locally, with an interpolant depending on the parametrization.

To avoid this problem, we first restrict our problem to convex symmetric domains, such that the boundary can be described globally through one convex function.

\begin{definition}[Convex half domains] \rm
Let $\Omega \subset \mathbb{R}^3$ be a convex symmetric domain with respect to the plane $\mathbb{R}^2 \times \{0\}$. We define  $\Omega^+ := \Omega \cap (\mathbb{R}^2\times \mathbb{R}_{\ge 0})$ as the \textit{ convex half domain,} which is a convex domain itself.
For a convex half domain $\Omega^+$ we define the two-dimensional domain $\omega \subset \mathbb{R}^2$ for which $\omega \times \{0\} = \Omega \cap (\mathbb{R}^2 \times \{0\})$, i.e. the projection of $\Omega$ to the plane of symmetry.
Let $\Gamma_{sym} := \omega \times \{0\}$ and  $\Gamma_{out} := (\partial \Omega)^+ = \partial \Omega \cap (\mathbb{R}^2\times \mathbb{R}_{\ge 0}) = (\partial \Omega^+ \cap \partial \Omega)$, such that $\partial \Omega^+ := \Gamma_{out} \cup \Gamma_{sym}$. \label{p1:def:chd}
\end{definition}
The domain $\Omega^+$ is half of a convex symmetric domain, if and only if we can describe it as the hypograph of a concave, non-negative function
$u : \omega \rightarrow \mathbb{R}$ defined on the plane $\omega = \Omega \cap (\mathbb{R}^2\times \{0\})$ so that  $\overline{\Omega^+} = \{(x,z) \in \mathbb{R}^2\times \mathbb{R}_{\ge 0}: x \in \omega, z \le u(x)\}$.

For example set $\Omega = B_1(0) \subset \mathbb{R}^3$. Then $\omega= \{(x,y)\in \mathbb{R}^2: x^2+y^2 \le 1\} \in \mathbb{R}^2$ and $u(x,y) = \sqrt{1-x^2-y^2}: \mathbb{\omega} \rightarrow \mathbb{R}$. The function $u$ is concave and $u \in H^1_{loc}(\omega)$. A visualization of the convex half domain is shown in Figure \ref{p1:fig:example_chd}.

 \begin{figure}[b]
\begin{center}
\includegraphics[height = 3.5cm]{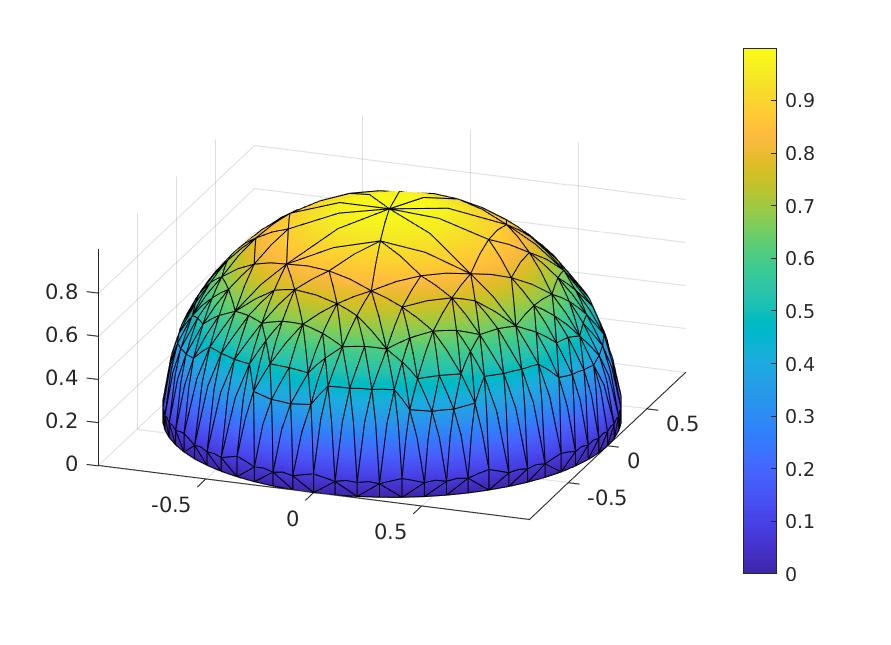}
\includegraphics[height = 3.5cm]{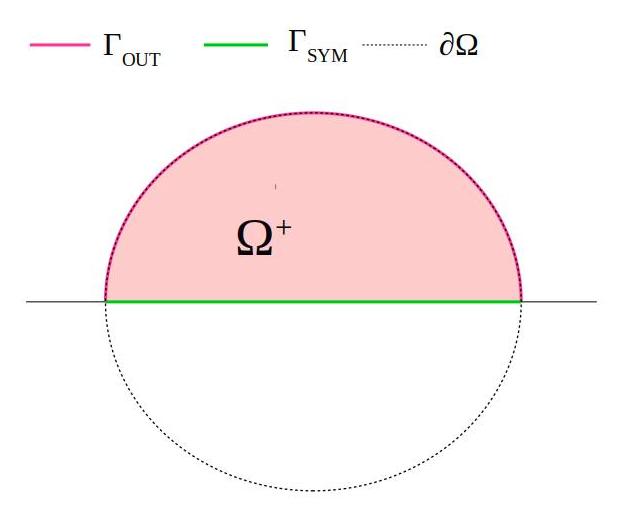}
  \caption{Approximation of the convex half domain $B_1(0)^+$ (left) and a two-dimensional sketch of the composition of the boundary as defined in Definition \ref{p1:def:chd} (right)} \label{p1:fig:example_chd}
\end{center}
     \end{figure}

Let $u: \omega \rightarrow \mathbb{R}$ be a concave function. Then $u$ is locally Lipschitz continous and $u \in H_{loc}^1(\omega)$, \cite[Theorem 2.1]{latorre}.
We assume here, that $u \in H^1(\omega)$, otherwise we approximate $\omega$ with a sequence $\omega_n \Subset \omega$ with a compact support in $\omega,$ such that $\chi_{\omega_n} \rightarrow \chi_{\omega}$  in $L^1$ as $n \rightarrow 0$. Then the functions $u_{\vert \omega_n}\in H^1(\omega_n)$ and the corresponding convex symmetric domains $\Omega_n$ satisfy $\chi_{\Omega_n} \rightarrow \chi_{\Omega}$ in $L^1$.\\ 
For the function $u \in H^1(\omega)$ we can find a sequence of concave functions $u_h \in \mathcal{P}^2(\mathcal{T}_h(\omega))$, such that $\Vert u_h- u \Vert_{H^1(\omega)} \rightarrow 0$ as $h \rightarrow 0$, see \cite{wachsmuth}.
The functions $u_h \in \mathcal{P}^2(\mathcal{T}_h(\omega))$ define convex half domains $\Omega_h$ with a piecewise quadratic boundary, such that
$\vert \Omega \triangle \Omega_{h} \vert \le c \Vert u_{h} -u \Vert_{L^2(\omega)}  \rightarrow 0 \text{ as } n \rightarrow \infty. $

We generate the discrete domain $\Omega_h^+$ with an  isoparametric  triangulation $\mathcal{T}^3_h$, by first generating a trianguluation of $\Gamma_{out}$ by computing the convex interpolant following \cite{wachsmuth}, and then generating a constraint Delaunay tetrahedralization of $\Omega_h^+$ using the \textsc{MATLAB} toolbox \textsc{GIBBON}, \cite{gibbon}.

\begin{remark}
We only describe here the approximation of symmetric convex domains. It is also possible to adapt the results to general convex bounded domains, since every bounded convex domain in $\mathbb{R}^3$ can be described as the intersection of the epigraph of a convex function and the hypograph of a concave function.
\end{remark}

\subsubsection{Local convexity constraint}

The results from \cite{wachsmuth} show us how to characterize the convexity, if we can describe the boundary globally by one convex function. For the shape optimization a local convexity constraint without a graph property is more convenient.
So in the following, if we approximate a symmetric convex domain directly, for example to obtain an initial domain for the shape optmization, we use the graph condition and the convex interpolant from \cite{wachsmuth}. 
For the shape optimization, in particular the constraint on the deformation field, we use a local convexity constraint, which allows us to avoid the graph condition.
\paragraph{Local constraint on the Hessian}
For the convex set $\Omega\subset \mathbb{R}^3$, any point $p \in \partial \Omega$ has a support plane, 
and if any point $p \in \partial \Omega$ has a support plane, and $\Omega$ is open, or a closed set with a non-empty interior, then $\Omega$ is convex, \cite[Theorem 1.3.3]{schneider}.  With a sufficiently smooth boundary the local convexity is equivalent to infinitesimal convexity, the positive semi-definiteness of the second fundamental form, \cite{bishop}.

Following \cite{demlow}, for a local parametrization $a: U \subset \mathbb{R}^2 \rightarrow \mathbb{R}^3$ of $\partial \Omega$, the second fundamental form with respect to the basis $\{a_{x_1},a_{x_2}\}$ of the tangent space is given by $II = [a_{x_{i,j}} \cdot \nu]$, with the outer unit normal $\nu = \frac{a_{x_1} \times a_{x_2}}{\vert a_{x_1} \times a_{x_2} \vert}$. 
For isoparametric domain $\Omega_h$, the interior of  $T^2 \in \mathcal{T}^2_h(\partial \Omega)$ of the triangulation of the boundary $\partial \Omega_h$ is smooth enough such that the local convexity is equivalent to the positive semi-definiteness of the second fundamendtal form
\begin{equation*}
II = [a_{T,x_{i,j}}, \nu_T]_{i,j} \succeq 0.
\end{equation*}
Due to the non-negativity of the norm, this is equivalent to 
\begin{equation} \label{p1:eq:c_h_simp}
C_H(T) = [a_{T,x_{i,j}}, a_{T,x_i} \times a_{T,x_j}]_{i,j} \succeq 0.
\end{equation}
This corresponds to the condition \eqref{p1:eq:hessian_condition} on the Hessian in \cite{wachsmuth}, and is only applied piecewise. Therefore we need a condition for the sides in the boundary triangulation, corresponding to the condition \eqref{p1:eq:kink_condition}  on the gradient jumps.

 \paragraph{Local constraint on the kinks}

For the kinks $S = T_1 \cap T_2$ of $\partial \Omega$ for two neighbouring simplices, we  derive a constraint, such that every point $p \in S$ has a support plane such that it is locally convex.
For $T_+,T_- \in \mathcal{T}^2_h(\partial \Omega_h)$  we denote the connecting edge as $S = \overline{T}_+\cap \overline{T}_-$. The curved triangles $T_+$ and $T_-$ are both $C^2$, but the union is only a Lipschitz surface.
 We assume that the curved triangles $T_+,T_-$ are convex, for example that they satisfy the infinitesimal convexity constraint \eqref{p1:eq:c_h_simp}.
 Due to the piecewise smoothness of the simplices we can assign to $x\in S$ two unit normals $\nu_+$ and $\nu_-$. The tangent planes defined by  $\nu_+$ and $\nu_-$ at point $x$ are each support planes for the simplices $T_+$ and $T_-$ respectively. We define the tangent vector of $x$ in $S$ as $\tau$.
 
We orient the simplices such that the tangent vector $\tau$ of $x$ at $S$ is so, that $T_+$ lies "left" of $S$ and $T_-$ lies "right" of $S$. With this orientation, we choose the normals of $S$ at $x$, lying in the corresponding tangent planes of $x$, i.e. $n^S_+ = \nu_+ \times \tau =  -\tau \times \nu_+$
and $n^S_- = \tau \times \nu_-$.
A sketch of the assumed orientation and the orientation of the outer normals for a convex and a non-convex kink in the boundary is shown in Figure \ref{p1:fig:kink_condition}.

We assume here, that the boundary is already piecewise locally convex, for example the triangles $T_+,T_-$ satisfy the infinitesimal convexity condition \eqref{p1:eq:c_h_simp}. Then it suffices to look at the local convexity for the tangent planes of $T_+, T_-$ at $x \in S$, since a support plane of the tangent planes is then also a support plane for the curved triangulations.
Therefore, without loss of generality it can be assumed that that $T_+$ and $T_-$ are affine triangles for any $x \in S$. 
The convexity condition on the kink in the isoparametric condition is then the same as a conformal convexity condition that can be imposed on a polyhedral approximation, see \cite[Section 5.3]{BW18}, which states that the normal vectors $\nu_+,\nu_-$ are oriented with respect to the vector $\tau$ such that  the signed volume $\det([\nu_-,\tau,\nu_+]) = \nu_- \cdot (\tau \times \nu_+)$ of the parallelepiped defined by the three vectors is non-negative. This implies that the local convexity at $x \in S$ is equivalent to the condition
\begin{equation}\label{p1:eq:c_k_prop}
C_K(x,S) = \nu_-(x) \cdot (\tau(x) \times \nu_+(x)) \ge 0.
\end{equation}

 \begin{figure}[b]
\begin{center}
\includegraphics[height = 4cm]{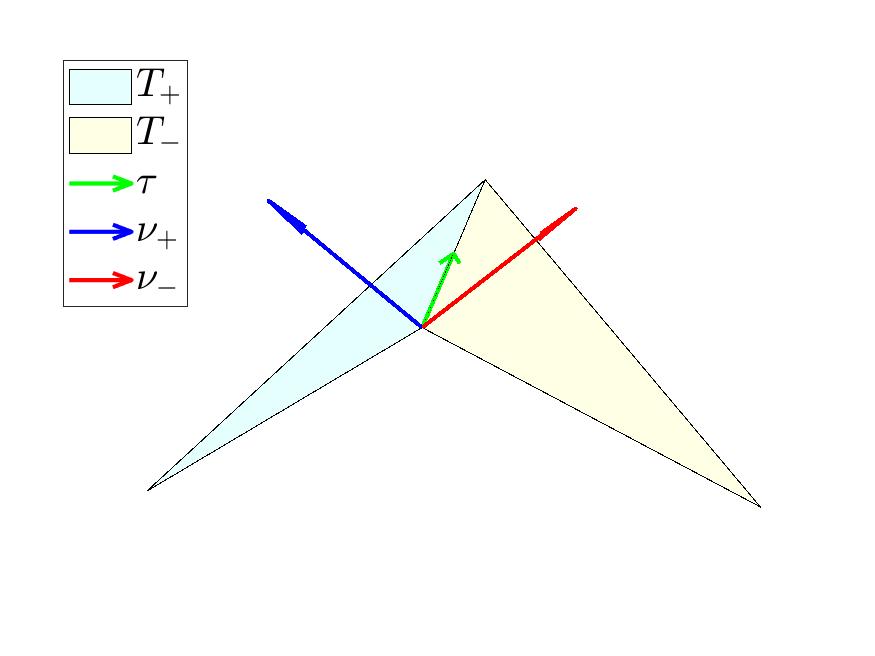}
\includegraphics[height = 4cm]{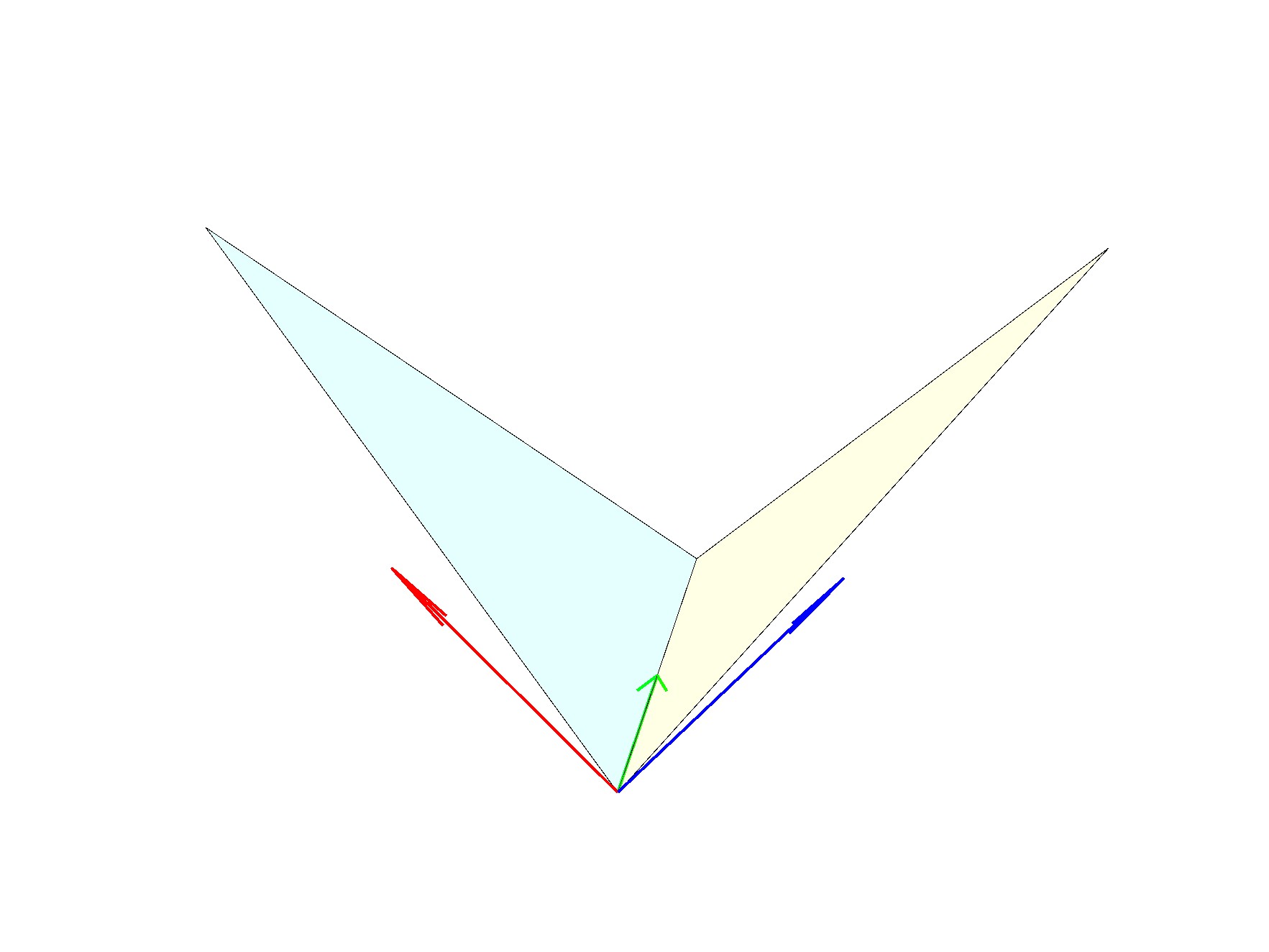}
  \caption{A sketch of the assumed orientation and the convexity condition for the kink between two neighboring simplices of a boundary triangulation of a domain $\Omega$ below the triangles, with a convex example (left) and a non-convex example (right). The convexity can by infered from the orientation of the outer normals $\nu_+,\nu_-$ and the tangent vector $\tau$ of the side $S$ connecting the simplices $T_+, T_-$.} \label{p1:fig:kink_condition}
\end{center}
     \end{figure}

\begin{remark}
If a neighbourhood $U\subset T_+ \cup T_-$ of $x$ can be described as the graph of a function, i.e. $U = \{(x,u(x)): x \in V\}$ for a domain open domain $V \in \mathbb{R}^{n-1}$, the outer unit normal is then given by
\begin{equation*}
\nu = \frac{(\nabla u,-1)}{\sqrt{1+\vert \nabla u \vert^2}}.
\end{equation*}
This can be used to derive the condition \eqref{p1:eq:c_k_prop} from the conditions \eqref{p1:eq:kink_condition} on the gradient jumps of \cite{wachsmuth}, if a suitable graph representation of the boundary is possible.
\end{remark}

Let $a^+: U_+ \subset \mathbb{R}^2 \rightarrow \mathbb{R}^3 $, $a^-: U_- \subset \mathbb{R}^2 \rightarrow \mathbb{R}^3$ be local parametrizations of $T_+, T_-$, and $a^S: U_S \subset \mathbb{R} \rightarrow \mathbb{R}^3$ a parametrization of the side $S$.
Due to the non-negativity of the norm, we then derive the equivalent condition to \eqref{p1:eq:c_k_prop} at the side $S$ of simplices $T_+$ and $T_-$ by
\begin{equation}\label{p1:eq:c_k_simp}
C_K(S)= (a^-_{x_1}\times a^-_{x_2}) \cdot \left( a^S_{x_1} \times\left( a^+_{x_1}\times a^+_{x_2}\right)\right) \ge 0 .
\end{equation}

In the case of convex half domains a modified kink condition is used to ensure the convexity of the mirrored domain
\begin{align*}
C^+_K(S)= \begin{cases} \nu_+^3 \ge 0 \quad &\text{ if }S \subset \Gamma_{\text{out}} \cap \Gamma_{\text{sym}}, \\
  C_K(S)  \quad &\text{ otherwise.} \end{cases}
\end{align*}

\subsection{Shape optimization for the Poisson problem} \label{p1:sec:isop_so}

As in Section \ref{p1:sec:non_conf}, we are interested in the shape optimization where the state equation is the Poisson equation.
\begin{align*}\label{p1:problem:P_D_cont}
\begin{cases}
\text{ Minimize} & \ \ \ \int_{\Omega}j(x,u(x),\nabla u(x)) \text{d} x \\ \text{ w.r.t.} & \ \ \ \Omega \subset \mathbb{R}^3, u \in H_0^1(\Omega) \tag{$\mathbf{P}$}\\ \text{ s.t.} & \ \ \ - \Delta u = f \text{ in } \Omega \text{ and } \Omega \subset Q \text{ convex, symmetric and open.}
\end{cases}
\end{align*}
Here, $Q \subset \mathbb{R}^3$  is bounded, open and convex, $j : Q \times \mathbb{R} \times \mathbb{R}^3 \rightarrow \mathbb{R}$ a suitable Carath\'{e}odory function and $f \in L^2(Q)$.

To reduce computational time, we only optimize among symmetric domains here, the algorithm can however also be adapted to general convex domains.  The symmetry constraint can be easily included into the existence result of Proposition \ref{p1:prop:pp_existence}. \\
For the discretization we consider the class $ \mathbb{T}$ of conforming, uniformly shape regular isoparametric triangulations $\mathcal{T}_h$ of isoparametric subset of $\mathbb{R}^3$ with elements $T \in \mathcal{T}_h$ with diameter $h_T \le h$. For the discretization we solve the following problem:
\begin{align*}
\begin{cases}
\text{ Minimize} & \ \ \ \int_{\Omega_h}j(x,u_h(x),\nabla u_h(x)) \text{d} x \\ \text{ w.r.t.} & \ \ \ \Omega_h \subset \mathbb{R}^3, \mathcal{T}_h \in \mathbb{T} \text{ isoparametric triangulation of } \Omega_h, u_h \in \mathcal{S}_0^{2,iso}(\mathcal{T}_h)\tag{$\mathbf{P_h}$}\\ \text{ s.t.} & \ \ \ - \Delta_h u_j = f_h \text{ in } \Omega_h \text{ and } \Omega_h \subset Q \text{ convex, symmetric and open.}
\end{cases}
\end{align*}

 Due to conformal approximation of the convexity constraint, it is straightforward to adapt the results from \cite{BW18} for the consistency and stability of the shape optimization using the approximation results for isoparametric domains from \cite[Chapter 4.4]{ciarlet}. 

\subsection{Shape optimization algorithm} \label{p1:sec:isop_alg}

The shape optimization works very similar to the one described in the previous Section \ref{p1:sec:non_conf}.
We again approximate optimal domains with a perturbation of identity, for which we first need to approximate the shape derivative. We approximate again the volume formulation of the shape derivative. This works well in the experiments, even if the
theory for approximate shape derivatives does only extend to affine linear deformations, \cite{hiptmair}.
We compute the deformation field from a problem of linear elasticity.

Further, we include the linearised convexity constraints derived from a first order expansion
\begin{align*}
C_K(\partial \Omega) + DC_K(\partial \Omega)\cdot V &\ge 0, \\
C_H(\partial \Omega) + DC_H(\partial \Omega)\cdot V &\succeq 0,
\end{align*}
defined earlier, see conditions \eqref{p1:eq:c_h_simp} and \eqref{p1:eq:c_k_simp}.
Together with the convexity constraints, the deformation field solves the following minimzation problem:
\begin{align*}
\begin{cases}
\text{Min }& \frac{1}{2} v_h^T a_h v_h = -J^\prime(\Omega_h,v_h) \\
\text{s.t. } & C_K(\partial \Omega_h) + DC_K \cdot V \ge 0 \\
 \text{and }&C_H(\partial \Omega_h) +DC_H \cdot V \succeq 0
\end{cases}
\end{align*}
In the case where we optimize only among convex half domains $\Omega^+$, the adapted convexity constraint $C^+_K$ is used, to ensure the convexity of the whole domain $\Omega$ near the plane of symmetry, and apply gliding boundary conditions on $\Gamma_{sym}$ for $V$.
In order to solve this problem and find the deformation field, we adapt the Augmented Lagrangian used in \cite{wachsmuth}.
As in Section \ref{p1:sec:non_conf} we include a line search to control the deformation, and iterate the steps until a suitable stopping criterion is satisfied.

\subsection{Numerical experiments} \label{p1:sec:isop_ex}

We repeat the experiments from Section \ref{p1:ldc:sec:num_ex}, but with the isoparametric approximation described in this section.  
\begin{align*}
\begin{cases}
\textup{Minimize } & J(\Omega^{+}) = \int_{\Omega^{+}} u(x) dx \\ 
\textup{w.r.t } & \Omega^{+} \subset Q \subset \mathbb{R}^2  \times \mathbb{R}_{\ge 0}\text{ a open convex half domain, } \,u \in H^1(\Omega^{+}) \\
\text{s.t. } & - \Delta u + u= f \text{ in } \Omega^{+}, \quad \frac{\partial u}{\partial n}	 = 0   \text{ on } \partial \Omega^{+}.
\end{cases}
\end{align*}
Here, we discretize the convex domains with convex isoparametric domains and consider again the following two functions:
\begin{align}
f_1 (x_1,x_2,x_3)& = x_1^2+x^2+x^3-1, \tag{$\mathbf{P^{+}_1}$} \label{p1:problem:P_1_half} \\
f_2 (x_1,x_2,x_3)& = 20(x_1+0.4-x_2^2)^2+x_1^2+x_1^2+x_1^2-1. \tag{$\mathbf{P^{+}_2}$} \label{p1:problem:P_2_half} 
\end{align}
Both problems are symmetric, and so are the solutions of the state equation due to the existence of a unique solution of the Poisson problem and the regularity results for the Poisson equation for convex domains. Therefore solving the shape optimization among convex half domains does not pose a restriction here, and we expect similar results for both the approximation with local discrete convex domains and isoparametric domains, except that the objective value differs by the factor 2.
As initial domains we choose the half ellipsoids $\Omega_1 = E^{+}_{0.9,0.9},\Omega_2= E^{+}_{1,1}$ and $\Omega_3 = E^{+}_{1.1,1.1}$, for different levels of refinements $h = 2^{-\ell}$.
As elasticity parameters we choose $E = 0.5,\nu =0.2$ for Young's Module, and the damping parameter $\rho = 0.5$, to ensure the coercivity of the bilinear form arising in the problem of linear elasticity.

The objective values  are listed in Table \ref{p1:table:optimal_PN_half} and the optimized domains are shown in Figure \ref{p1:fig:optimal_P1N_half} and \ref{p1:fig:optimal_P2N_half} for problems \ref{p1:problem:P_1_half} and \ref{p1:problem:P_2_half} respectivly.
For $f_1$, the algorithm approximated the domain accuratley for different initial domains. 
Following our expectations and the previous numerical experiments in Section \ref{p1:ldc:sec:num_ex}, the optimized domains are approximations of the unit ball.
For $f_2$ the algorithm also performed well, with a similar violation of the convexity constraint as for problem \ref{p1:problem:P_1_half} and the approximated optimal domains and objective values are consistent with the approximated optimal domains with the local discrete convexity constraint, see Section \ref{p1:ldc:sec:num_ex}.
For both experiments only the refinement $\ell \le 2$ was possible, due to the higher computational effort using $\ell > 2$.

\begin{table}[!t]
\caption{Objective values  $J(\Omega^{+}_i)$ of approximated optimal domains $\Omega^{+}_i$ for different initials domains $\Omega^+_{0,i}$ with break condition $\tau$, number of iterations $k$, and convexity conditions $C_H, C_K$ of approximated domains for problems \eqref{p1:problem:P_1_half} and \eqref{p1:problem:P_2_half} (bottom) and refinement level $\ell = 2,h = 2^{-\ell}$; approximated optimal domains are shown in  Figure \ref{p1:fig:optimal_P2N_half}}
\label{p1:table:optimal_PN_half}
\begin{center}
\begin{tabular}{p{0.6cm} p{1.1cm}p{1.1cm}p{1.1cm}p{1.1cm}p{1.1cm}p{2.2cm}p{2.2cm}p{0.4cm}}
\hline\noalign{\smallskip}
 \eqref{p1:problem:P_1_half} & $J_1(\Omega^{+}_{0,i}) $& $J_1(\Omega^{+,}_i)$ & $\vert\Omega^{+}_{i,0}\vert $&$ \vert\Omega^{+}_i \vert $ &$ \tau$  &$ \min(C_K(\Omega^{+}_i))$ & $\min(C_H(\Omega^{+}_i))$& $k$  \\ \noalign{\smallskip}\hline\noalign{\smallskip}
$\Omega_1$   &-0.7149&	-0.7764	&2.0836&	2.0709&	-5.5e-5&	-5.6681e-7&	-3.0954e-6&12\\
$\Omega_2$ &-0.7868&	-0.7874	&2.1131&	2.1036&	-5.5e-6&	-7.8455e-8&	-1.2895e-6&	4  \\
$ \Omega_3$&-0.7444&	-0.7837	&2.0817	&2.0841&	-1.8e-5&	-3.8683e-7&	-3.7264e-6	&11 \\
\noalign{\smallskip}\hline\noalign{\smallskip}
& & & & & & &  \\
\hline\noalign{\smallskip}
 \eqref{p1:problem:P_2_half} & $J_2(\Omega^{+}_{0,i}) $& $J_2(\Omega^{+,}_i)$ & $\vert\Omega^{+}_{i,0}\vert $&$ \vert\Omega^{+}_i \vert $ &$ \tau$  &$ \min(C_K(\Omega^{+}_i))$ & $\min(C_H(\Omega^{+}_i))$& $k$ \\ \noalign{\smallskip}\hline\noalign{\smallskip}
$ \Omega_1$&9.9534	&-0.1003&	2.0836	&0.2977&	-5.5e-5	&-5.5392e-7	&-1.3406e-6&	83
 \\ $ \Omega_2$&11.8133&	-0.1024	&2.1131	&0.2988&	-3.8e-4	&-1.4899e-7&	-1.1170e-6&	78\\
 $\Omega_3$&13.5209&	-0.1019	&2.0817	&0.3049	&-6.9e-5&	-4.5366e-7&	-5.2413e-6&73 \\ 
\noalign{\smallskip}\hline\noalign{\smallskip}
\end{tabular}
\end{center}
\end{table}

	 \begin{figure}[h]
\begin{center}
\includegraphics[height = 3.5cm]{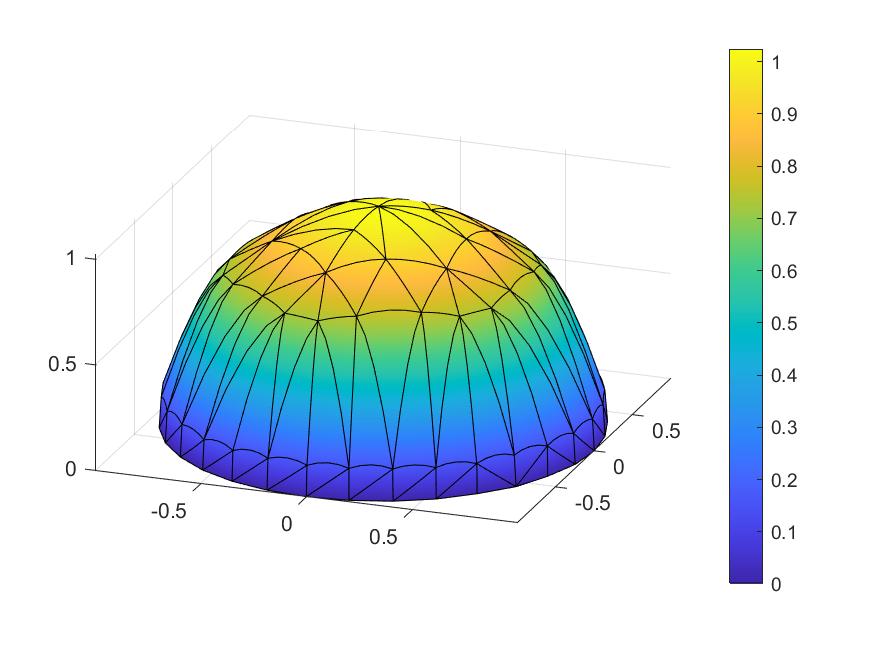}
\includegraphics[height = 3.5cm]{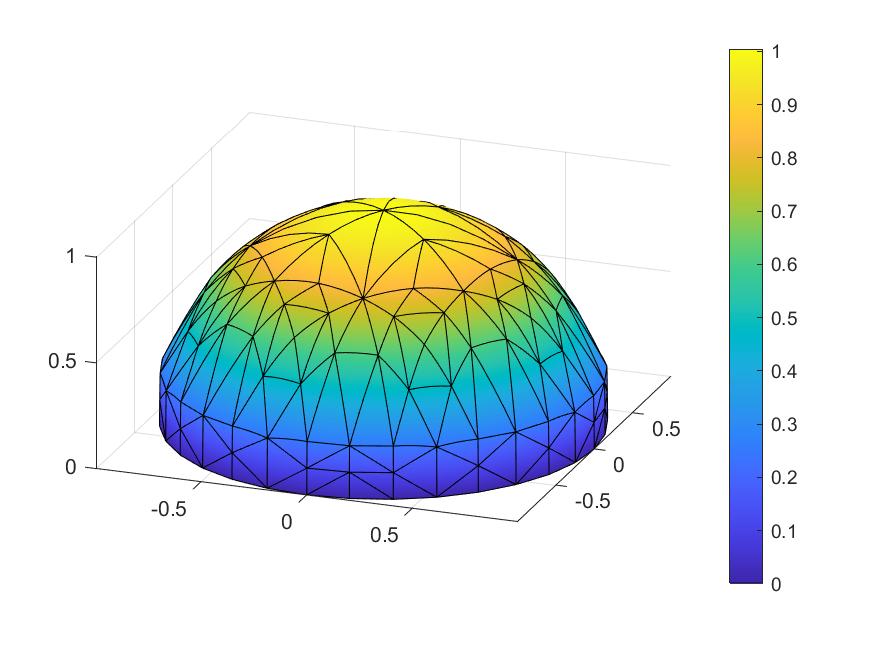}
\includegraphics[height = 3.5cm]{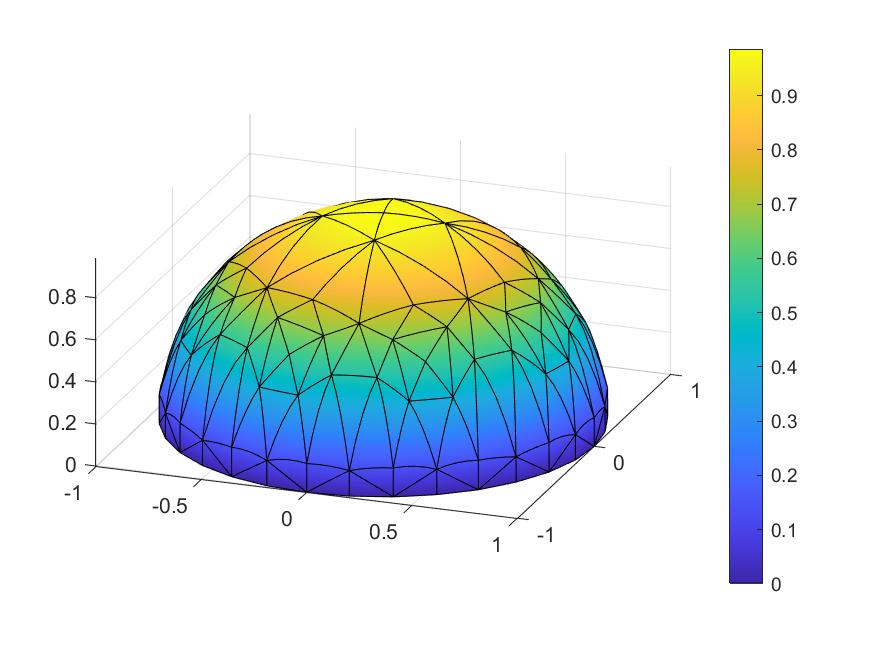}\
  \caption{Approximated optimal domains for \eqref{p1:problem:P_1_half}; with  initial domains $E^+_{0.9,0.9},E^+_{1,1},E^+_{1.1,1.1}$ (left to right), see Table \ref{p1:table:optimal_PN_half} (top) for the corresponding objective values} \label{p1:fig:optimal_P1N_half}
\end{center}
     \end{figure}

	 \begin{figure}[h]
\begin{center}
\includegraphics[height = 3.5cm]{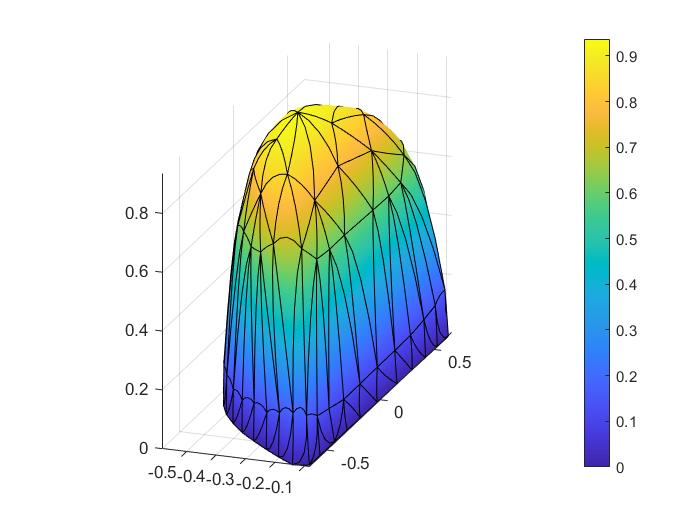}
\includegraphics[height = 3.5cm]{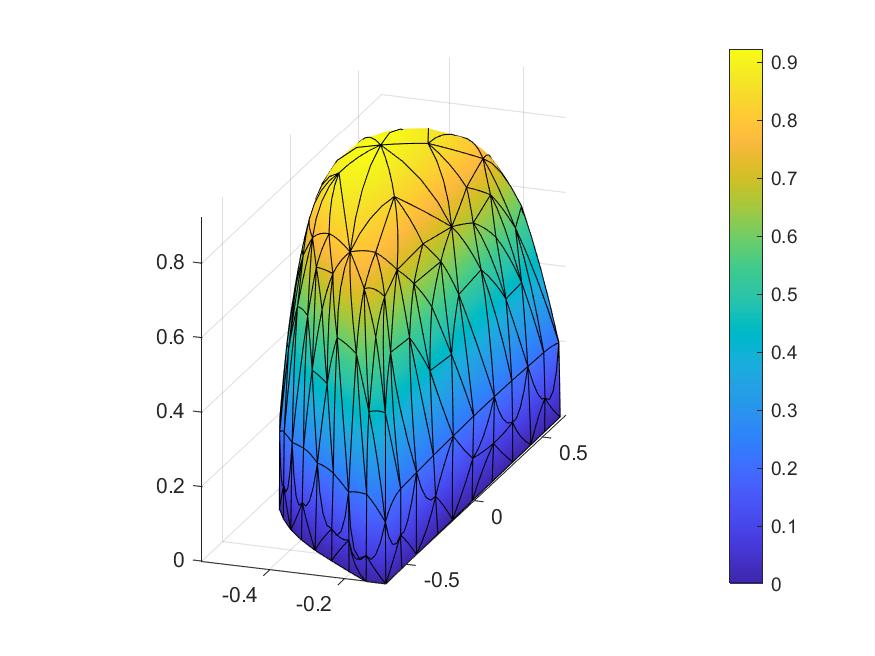}
\includegraphics[height = 3.5cm]{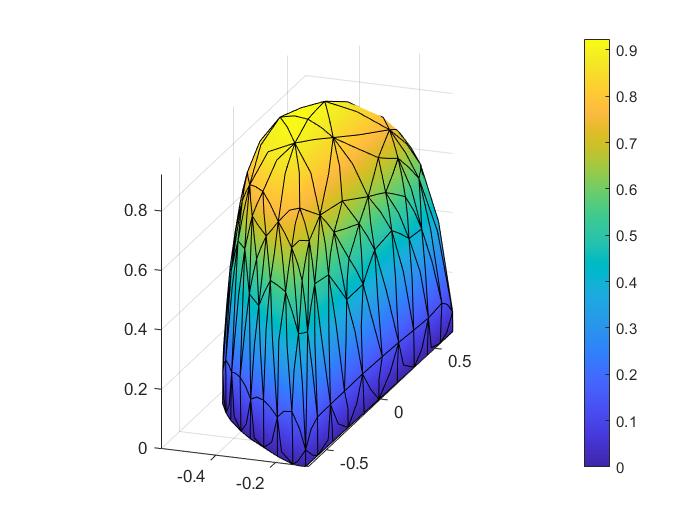}\\
  \caption{Approximated optimal domains for \eqref{p1:problem:P_2_half}; with  initial domains $E^+_{0.9,0.9},E^+_{1,1},E^+_{1.1,1.1}$(left to right), see Table \ref{p1:table:optimal_PN_half} (bottom) for the corresponding objective values} \label{p1:fig:optimal_P2N_half}
\end{center}
     \end{figure}

\section{Conclusion}\label{p1:sec:conclusion}

We have proposed two methods for the approximation of optimal convex domains in $\mathbb{R}^3$, which extend the results from \cite{keller} for rotational symmetric domains.

In Section \ref{p1:sec:non_conf} we proposed a relaxed definition of convexity for affine linear tetrahedralizations, which provides accurate results in the investigated shape optimization problems constrained by a Poisson equation. The definition of discrete convexity gives an approximation of the volume of a convex domain, but the regularity of the boundary of a convex set is lost, which makes it difficult to use this discretization for shape optimization problems in which boundary terms occur.  Recent results show that under further restrictions on the class of triangulations, problems such as the optimization of an eigenvalue arising in a problem of optimal insulation, see \cite{BB19,keller,BBN17}, can also be approximated. 
The convergence results for classes of convex domains are also available to other classes of domains, such as classes of domains which are uniform Lipschitz, or classes of domains with uniformly bounded perimeters, \cite[Theorem 6.3 and 6.11]{delzol}. Whether this provides stability in the shape optimization remains to be investigated.
 \\
In Section \ref{p1:sec:chd} an approximation of convex domains with isoparametric finite elements with a piecewise quadratic boundary was proposed, which allows for a conformal approximation of the convexity constraint. 
Here, we only optimized among symmetric convex domains, however the algorithm can be adapted to general convex domains without additional geometric constraints.
For the problems under consideration, both methods terminated within less than 150 iterations and the numerical experiments gave consistent results.
In both cases it was only possible to run the algorithms for a few different levels of mesh refinements, due to the high computational effort associated with the discretizations in higher dimensions. 
This is problematic, when considering more complex problems, for example when a volume constraint is added. For the approximation with discrete convex domains solving the deformation field as a solution of a stationary Stokes system to maintain the volume requires the use of Crouzeix-Raviart finite elements. This adds an additional discretization error and computational effort which makes the use of this method less practical. \\
For the isoparametric approximation including a volume constraint by using a Taylor Hood-like method as in \cite{bazilevs,bressan_juettler} gave promising results for simple experiments.

\addcontentsline{toc}{part}{References}

\printbibliography[title = References]

\end{document}